\documentclass[11pt]{amsart}
\usepackage{geometry}                
\usepackage{graphicx}
\usepackage{amssymb}
\usepackage{epstopdf}
\usepackage{ifpdf}
\usepackage{color}
\usepackage{tikz}
\usepackage{diagrams}
\usepackage{dsfont}
\newarrow{Funct}{}{dash}{}{dash}{>}
\definecolor{dmagenta}{rgb}{.5,0,.5} 
\definecolor{dred}{rgb}{.5,0,0} 
\definecolor{dgreen}{rgb}{0,.5,0} 
\definecolor{blue}{rgb}{0,0,0.5} 
\definecolor{black}{rgb}{0,0,0} 
\definecolor{vdgreen}{rgb}{0,.3,0} 
\definecolor{vdred}{rgb}{.3,0,0} 
\definecolor{red}{rgb}{1,0,0} 

\newcommand{\lplus}{{\mathfrak{h}}}        

\newcommand\cO{{\mathcal{O}}} 
\newcommand\cB{\mathcal{B}}    

\newcommand\cLO{{\mathcal{LO}}}    
\newcommand\cL{\mathcal L}             

\newcommand{\Assoc}{\mathcal{A}ssoc}  
\newcommand{\Lie}{\mathcal{L}ie}  
\newcommand{\Com}{\mathcal{C}om}  
\newcommand{\Q}{{\mathbb{Q}}}
\newcommand{\R}{{\mathbb{R}}}
\newcommand{\F}{{\mathds{k}}}
\newcommand{\Z}{{\mathbb{Z}}}
\newcommand{\C}{{\mathbb{C}}}

\newcommand{\hairy}{\mathcal H}  

\newcommand{\phairy}{\mathcal{PH}} 
\newcommand{\G}{{\bf G}}              
\newcommand{\ext}{\bigwedge\nolimits}
\newcommand{\bdry}{\partial}
\newcommand{\Wedge}{\bigwedge}
\newcommand{\id}{\mathrm{Id}}   
\DeclareMathOperator{\Tr}{Tr}  
\DeclareMathOperator{\SP}{Sp}  
\DeclareMathOperator{\Out}{Out} 

\DeclareMathOperator{\GL}{GL}  
\DeclareMathOperator{\Aut}{Aut}  
\DeclareMathOperator{\Mod}{Mod}  
\DeclareMathOperator{\im}{im}  
\DeclareMathOperator{\SL}{SL}
\let\ker\undefined
\DeclareMathOperator{\ker}{ker}  
\DeclareMathOperator*{\invlim}{\varprojlim\rule[-2ex]{0ex}{4.5ex}}  
\DeclareMathOperator*{\dirlim}{\varinjlim\rule[-2ex]{0ex}{4.5ex}}  
\DeclareMathOperator{\spf}{\mathfrak{sp}}

\DeclareMathOperator{\Hom}{Hom}
\newcommand{\funct}{\rightsquigarrow}
\newcommand{\SF}[1]{{\mathbb S}_{#1}}  
\newcommand{\SpF}[1]{{\mathbb S}_{\langle#1\rangle}}  
\newcommand{\degree}[1]{{[\![#1]\!]}}

\newcommand{\dash}{{\mbox{--}}}

\newtheorem{proposition}{Proposition}[section]
\newtheorem{definition}[proposition]{Definition}
\newtheorem{theorem}[proposition]{Theorem}
\newtheorem{lemma}[proposition]{Lemma}

\newtheorem{theoremB}{Theorem}
\newtheorem{theoremC}{Theorem}
\newtheorem{theoremA}{Theorem}
\newtheorem{theoremD}{Theorem}
\newtheorem{theoremE}{Theorem}
\newtheorem{theoremF}{Theorem}
\theoremstyle{remark}

\newtheorem{remark}[proposition]{Remark}
\newtheorem*{example}{Example}
\newtheoremstyle{red}{3pt}{3pt}{\color{red}}{}{\itshape}{.}{.5em}{}
\theoremstyle{red}

\title{Hairy graphs and the unstable homology\\ of $\Mod(g,s), \Out(F_n)$ and $\Aut(F_n)$\\
}
\author{Jim Conant}
\author{Martin Kassabov}
\author{Karen Vogtmann}

\begin{document}
\begin{abstract}
We study a  family of Lie algebras $\lplus{\cO}$ which are defined for cyclic operads $\cO$.  Using his graph homology theory, Kontsevich identified the homology of two of these Lie algebras (corresponding to the Lie and associative operads)  with the cohomology of outer automorphism groups of free groups  and mapping class groups of punctured surfaces, respectively.  In this paper we introduce  a {\it hairy graph homology} theory for $\cO$.  We show that the homology of $\lplus\cO$ embeds in hairy graph homology via a {\it trace} map which generalizes the trace map defined by S. Morita.  For the Lie operad we use the trace map to find large new summands of the abelianization of $\lplus{\cO}$  which are related to classical modular forms for $\SL_2(\Z)$.  Using cusp forms we construct new cycles for the unstable homology of $\Out(F_n)$, and using Eisenstein series we find new cycles for $\Aut(F_n)$.  For the associative operad we compute the first homology of the hairy graph complex by adapting an argument of Morita, Sakasai and Suzuki, who determined the complete abelianization of $\lplus\cO$ in the associative case. 

\end{abstract}
\maketitle
\section{Introduction}
\label{sec:Introduction}

A wide variety of mathematical and physical phenomena can be modeled using trees and graphs, and it is not surprising that structures based on trees and graphs play an important role in analyzing these phenomena. One such structure is the Lie algebra $\lplus_V$ generated by finite planar trivalent trees whose leaves are decorated with elements of a symplectic vector space $V$ (see Section~\ref{sec:review} for a precise definition of $\lplus_V$.)   
When one stabilizes by letting the dimension of $V$ go to infinity, the resulting object $\lplus_\infty$  is a fascinating and complicated Lie algebra, which arises naturally in several areas of topology and geometric group theory.  These include  the study of the Johnson filtration of the mapping class group of a surface \cite{J,Mor},    finite-type invariants of 3-manifolds  \cite{CST, GL,L}, and   the homology of outer automorphism groups of free groups  \cite{CV, Ko2,Ko1}.

An important first step in understanding  $\lplus_\infty$ is to determine its abelianization. S. Morita \cite{Morita} constructed a large abelian quotient of $\lplus_V$ via his {\it trace map}, which takes values in the polynomial algebra   $\F[V].$    He  conjectured that the image of the stabilized trace map is isomorphic to the entire abelianization of $\lplus_\infty$~\cite[Conjecture 6.10]{Morita}. However, in this paper we will show that  the abelianization is in fact quite a bit larger than this, with much richer structure than anticipated.   One use Morita made of his trace map was to construct a series of elements of $H^*(\lplus_\infty)$, which he then identified with elements of   $H_*(Out(F_n))$ using a theorem of M. Kontsevich \cite{Ko2,Ko1}.  The new parts we find of the abelianization of $\lplus_\infty$ allow us to construct many new homology classes for $Out(F_n)$.

Morita's trace can be re-interpreted in terms of graphs, as was done in \cite{CVMorita}; the natural target of the trace is then a vector space spanned by graphs consisting of one oriented cycle with some ``hairs" attached, labeled by elements of $V$.  This point of view leads to the central construction of this paper, which is an extension of the trace map to take into account graphs of higher rank.   The domain of this extension is the Chevalley-Eilenberg chain complex $C_\bullet(\lplus_V)$, and the target  is a ``hairy graph" complex $\mathcal H_V$, spanned by oriented graphs with labeled hairs attached (see section~\ref{complex}).   

Our first main theorem is that after stabilization, this new trace map  captures all of the abelianization of $\lplus_\infty$ and in fact does even better, capturing the higher homology groups as well: 

 \begin{theoremA}\label{thm:intro-injective}
 $\Tr\colon C_\bullet(\lplus_V)\to \mathcal H_V$  is a chain map which induces an injection on homology $H_*(\lplus_\infty)\hookrightarrow H_*(\mathcal H_\infty).$
 \end{theoremA}

 The trace map is not surjective on homology, even after stabilization, but the image is large in a precise sense (Theorem~\ref{thm:surjective}), so the next task in determining the abelianization of $\lplus_\infty$ is to analyze $H_1(\mathcal H_V)$.  This breaks up as a direct sum according to the ranks $r$ of the  hairy graphs:  $$H_1(\mathcal H_V)=\bigoplus_{i=0}^\infty H_{1,r}.$$ It is easy to check that $H_{1,0}\cong \ext^3V$ and $H_{1,1}\cong\bigoplus_{k=0}^\infty S^{2k+1}(V)$, and to then confirm that  this part is equal to the image of  Morita's original trace map.  But higher-rank graphs also make large contributions to the abelianization:
 \begin{theoremB}
For $r\geq 2$ there is an isomorphism $$H_{1,r}\cong H^{2r-3}(\Out(F_r);\F[V^{\oplus r}]),$$ where $\F[V^{\oplus r}]$ denotes the polynomial algebra generated by $V^{\oplus r}$, and $\Out(F_r)$ acts on $\F[V^{\oplus r}]=\F[V\otimes \F^{r}]$ via the natural $\GL(r,\Z)$ action on $\F^r$.
 \end{theoremB}
 These theorems together say that  the abelianization of $\lplus_\infty$ is a large subspace of a direct sum of certain twisted cohomologies of $\Out(F_r)$.  We next show that the first new summand $H_{1,2}$ of $H_1(\mathcal H_V)$ can be  completely calculated using the Eichler-Shimura isomorphism, which  relates classical modular forms with twisted cohomology of $\SL_2(\Z)$.

  \begin{theoremC}\label{intro:thm-rank2} Let $s_k$ denote the dimension of the space of classical weight $k$ cusp forms for $\SL_2(\Z)$.  Then 
$$H_{1,2}\cong \bigoplus_{k>\ell\geq 0} (\SF{(k,\ell)}V)^{\oplus \lambda_{k,\ell}},
$$
where $\SF{(k,\ell)}$ is the Schur functor for the partition $(k,\ell)$ and $\lambda_{k,\ell}=0$ unless $k+\ell$ is even, in which case $$\lambda_{k,\ell}=\begin{cases}s_{k-\ell+2} &\text{if }\ell \text{ is even}\\
s_{k-\ell+2}+1&\text{if }\ell \text{ is odd}.
\end{cases}$$
\end{theoremC}

Preliminary calculations indicate that $H_{1,3}$  is also highly nontrivial and the dimensions appear related to modular forms, but we will defer these calculations to another   paper.

For our main applications of Theorem~\ref{intro:thm-rank2} we use the seminal work of M. Kontsevich~\cite{Ko2,Ko1}, which relates the cohomology of $\lplus_\infty$ to the homology of outer automorphism groups of free groups.  As Morita showed, wedge products of elements of the abelianization of $\lplus_\infty$ can be pulled back to produce cocycles for $\lplus_\infty$, which via  Kontsevich's theorem give rise to  cycles in a chain complex for $\Out(F_n)$.  It is remarkable that in this way classes in the twisted cohomology of $\Out(F_2)$ produce untwisted rational homology classes for $\Out(F_n)$ for infinitely many values of $n$.  For example, we show how Theorem~\ref{intro:thm-rank2} allows us to produce cycles for $\Out(F_n)$ based on spaces of cusp forms:
 
 \begin{theoremD}\label{thm:intro-symmod} 
  There is an embedding
$$
\ext^2\left(M^0_{2k}\right)^*\hookrightarrow Z_{4k-2}(\Out(F_{2k+1});\Q)
$$
into cycles for $\Out(F_{2k+1})$, where $M^0_{2k}$ is the vector space of cusp forms for $\SL(2,\Z)$ of weight $2k$.
\end{theoremD}

Using a variant of Kontsevich's method, Gray~\cite{Gray} has related the homology of $\Aut(F_n)$ to a certain twisted cohomology of $\lplus_\infty$, and as another application of  Theorem~\ref{intro:thm-rank2} we use his work to produce cycles based on Eisenstein series: 
 
 \begin{theoremE}
There is a   series of cycles  $$e_{4k+3}\in Z_{4k+3}(\Aut(F_{2k+3});\Q)$$ which are  related to Eisenstein series. 
 \end{theoremE}
 
 With the help of a computer
we have shown that  the first two of these classes, $e_7$ and $e_{11}$, represent nontrivial homology classes. The class  $e_7\in H_7(\Aut(F_5);\Q)$ coincides with a class found
several
years earlier by F. Gerlits~\cite{Gerlits}, by quite different methods.  This class had not previously seemed to fit into the picture of classes coming from the abelianization of $\lplus_\infty$, but now does.  In fact, at this point all known rational homology classes for $\Aut(F_n)$ and $\Out(F_n)$, of which there are only a handful, arise from the abelianization of $\lplus_\infty$.

The Lie algebra $\lplus_V$ discussed above is actually just one  example of a much more general construction.   
A labeled tree in $\lplus_V$  can be thought of as an element of the Lie operad which has vectors labeling its input/output slots. In a similar way one can define a Lie algebra  $\lplus\cO_V$ for any cyclic operad $\cO$ (see Section~\ref{sec:review}). The graphical trace map $\Tr$ is also defined in this  generality, and takes values in a hairy version of $\cO$-graph homology (Sections~\ref{sec:hairy_graphs} and ~\ref{sec:Trace_map}). This $\cO$-graph  homology theory $\mathcal H\cO_V$ is similar to that defined in \cite{CV}, except that the graphs are allowed to have univalent vertices labeled by  vectors in  $V$.  

This more general point of view is the one we take throughout this paper, and Theorem~\ref{thm:intro-injective} is actually stated and proved as follows:

\begin{theoremF}
For any cyclic operad $\cO$ which is finite-dimensional in each degree, $\Tr$ ids a chain map which, after stabilization with respect to $V$, induces an injection $H_*(\lplus\cO_\infty)\hookrightarrow H_*(\mathcal H\cO_\infty).$   
\end{theoremF}

For the associative operad, Kontsevich showed that  the homology of $\lplus\cO_\infty$ is related to the  rational homology of the mapping class group \cite{CV, Ko2,Ko1}. In section~\ref{sec:assoc} we adapt an argument of Morita, Sakasai and Suzuki \cite{mss} to compute the first homology group of the hairy $\cO$-graph complex in the case $\cO=\Assoc$.  Their argument is part of their result that the abelianization of $\lplus\cO_\infty$ is precisely equal to the piece determined by Morita in~\cite{Morita2}, and in Theorem~\ref{thm:associative} we show how this also follows from the first homology computation together with injectivity of the trace map.

 We remark that the stable homology of both $\Mod(g,s)$ and $\Out(F_n)$ is well understood (see~\cite{Galatius, MW}) but the unstable homology remains quite mysterious, and all the classes we find in this paper lie in the unstable range.

Finally, in section \ref{sec:thornedgraphs} we note that hairy Lie graph homology is related to  the cohomology of mapping class groups of certain punctured $3$-manifolds, as defined in \cite{HV2}.  Specifically, let $M_{n,s}$ be the compact 3-manifold obtained from the connected sum of $n$ copies of $S^1 \times S^2$ by deleting the interiors of s disjoint balls, and let $\Gamma_{n,s}$  be the quotient of the mapping class group of $M_{n,s}$ by the normal subgroup generated by Dehn twists along embedded $2$-spheres. Then hairy Lie graph homology is closely related to the cohomology of $\Gamma_{n,s}$  (Theorem~\ref{thm:Aut-Out}).

This paper is the first in a series of two papers.  In this paper we have concentrated on the theory needed to understand the abelianization of $\lplus\cO$ and applications to the homology of mapping class groups and automorphism groups of free groups.  In the sequel we will obtain a precise description of the image of the trace map  and show  that the $\SP$-module decomposition of the image corresponds exactly to the $\GL$-module decomposition of hairy graph homology.
  We will also
show how to use higher hairy graph homology to produce classes in the cohomology of $\lplus\cO$, yielding even more potential unstable homology classes for $\Aut(F_n)$, $\Out(F_n)$ and $\Mod(g,s)$.  Finally,
we will explain connections between hairy graph homology and Getzler and Kapranov's theory of modular operads, and also with Loday's dihedral homology.

\subsection*{Acknowledgements:} The authors wish to thank Francois Brunault for locating the reference \cite{Hab}, Shigeyuki Morita for pointing out an error in an earlier version of Lemma~\ref{lemma:SpCalculation} and the referee for several very useful comments. Jonathan Gray was instrumental in the calculation that $e_{11}\neq 0$ in Theorem~\ref{thm:autclasses}. The first author thanks Naoya Enomoto for answering a relevant MathOverflow post. The first author was supported by NSF grant DMS-0604351, the second author was supported  by NSF grant DMS-0900932 and the third author was supported by NSF grant DMS-1011857.

\section{Review of  the Lie algebra associated to a cyclic operad and its (co)homology}\label{sec:review}

All vector spaces in this paper will be over a fixed field $\F$ of characteristic $0$ and have either finite or countable dimension.  In this section, we also fix  a cyclic operad $\cO$ in the category of $\F$-vector spaces. Let $\cO((m))$ denote the vector space spanned by operad elements with $m$ input/output slots (any one of which can serve as the output for the other $(m-1)$ inputs). By definition the symmetric group $\Sigma_m$ acts on $\cO((m))$.
We will assume that the vector spaces $\cO((m))$ are finite-dimensional for each $m$, and we fix a basis for each $\cO((m))$.

\subsection{The Lie algebra $\cL_V=\cLO_V$}
\label{sec:Lie_algebra}\label{Lie}
We recall from~\cite{CV} how to construct a Lie algebra from $\cO$ and a
symplectic vector space $(V,\omega)$.  It will often be convenient to specify a symplectic basis $\cB$ for $V$.
Our main example will be the $2n$-dimensional vector space $V_n$ with the standard symplectic form and basis $\cB_n=\{p_1,\ldots,p_n,q_1,\ldots,q_n\}$, i.e. the matrix of  $\omega$ in the basis $\cB_n$ is
$
\begin{pmatrix}
0&I\\
-I&0
\end{pmatrix}.
$
We will also often consider the direct limit $V_\infty$ of the $V_n$ with respect to the natural inclusion, with basis
$\cB_\infty$.
For $x\in \cB_\infty$ the dual $x^*$ is given by  $p_i^*=q_i$ and $q_i^*=-p_i$.

\begin{definition} An {\em $\cO$-spider} is an operad element whose input/output slots $\lambda$
(called {\it legs}) are each decorated with an element $x_\lambda\in V$.
If the operad element is a basis element of $\cO$ and the labels are in  $\cB$, the spider is called \emph{a basic $\cO$-spider}.
\end{definition}

The Lie algebra
$\cLO_V$ is generated (as a vector space) by $\cO$-spiders; in particular the basic spiders generate $\cLO_{V}$.  Unless we need to specify the operad we will  denote $\cLO_V$ simply by $\cL_V$, and $\cLO_{V_n}$ by $\cL_n$.

There is a very useful grading on $\cL_V$ given as follows:

\begin{definition} The {\em degree} of an $\cO$-spider is the number of legs minus 2.
\end{definition}

We denote by $\cL_V\degree{d}$ the subspace of $\cL_V$   generated by spiders of degree $d$.  More formally, we have  $$\cL_V\degree{d}\cong \cO((d+2))\otimes_{\Sigma_{d+2}} V^{\otimes (d+2)}$$ and
$$\cL_V=\bigoplus_{d\geq 0} \cL_V\degree{d}.$$

Two spiders $s_1$ and $s_2$ can be {\it fused} into a single spider by using a leg $\lambda_1$ from the first as output and a leg $\lambda_2$ of the second as input and performing operad composition.  The unused legs retain their labels, and the resulting $\cO$-spider is multiplied by the symplectic pairing $\omega(x_{\lambda_1}, x_{\lambda_2})$ of the labels.  We denote the end result of this  by
$$
(s_1\cdot s_2)_{(\lambda_1,\lambda_2)}.
$$
Note that if each $s_i$ is a basic spider, the coefficient of the result is either $0$ or $\pm 1$.

We now define
$$
[s_1,s_2]=\sum_{{e}\in L_1\times L_2}  (s_1\cdot s_2)_{e}
$$
where $L_i$ is the set of legs of $s_i$.
This bracket gives $\cL_V$ the structure of  a Lie algebra: antisymmetry of the bracket follows from antisymmetry of the symplectic
form and the Jacobi identity is a consequence of associativity of
composition in the operad.

\subsection{The symplectic action}

Note that the degree of spiders is additive under bracket.  In particular, spiders of degree 0 based on the identity element of the operad $\cO$ generate a Lie subalgebra of $\cL_V$, and spiders of positive degree generate another Lie subalgebra, denoted $\lplus\cO_V$, or simply $\lplus_V$ if the operad is understood.   The subalgebra generated by degree 0 spiders based on the identity element is isomorphic to $\spf_V$,
and it acts on $\cL_V$  via the adjoint action, preserving degree.   In particular, the action of $\spf_V$ restricts to an action on  $\lplus_V$.

The symplectic group
 $\SP_V$ also acts on   $\cL_V$ (and on $\lplus_V$) by acting on the leg-labels.
The elements  of $\cL_V$ which are fixed by the $\SP_V$-action
are precisely those that are killed by the $\spf_V$ action.
These are called the {\it invariants} of the action.

The natural inclusion $V_n\hookrightarrow V_{n+1}$ induces an inclusion
$\lplus_n\hookrightarrow \lplus_{n+1}$, and stabilizing as $n\to \infty$
one obtains
$\lplus_\infty=\dirlim \lplus_n\rule[-2ex]{0ex}{5ex}$.

We will be principally concerned with the homology of the positive degree subalgebra
$\lplus_\infty$ of $\cL_\infty$
since in the cases $\cO=\Assoc$ and $\cO=\Lie$ it is the (primitive part of the symplectic invariants of) this homology which computes the cohomology of $\Out(F_n)$ and mapping class groups.

\subsection{Chevalley-Eilenberg differential}
The Lie algebra homology of $\lplus_V$ with trivial coefficients in $\F$ is computed using the exterior algebra $\ext\lplus_V$ and the Chevalley-Eilenberg differential
\begin{align*}
\bdry_{\text{Lie}}(s_1\wedge\ldots\wedge s_k)&=
\sum_{i<j} (-1)^{i+j+1}[s_i,s_j] \wedge s_1 \wedge\ldots \hat s_i\wedge \ldots \wedge \hat s_j\wedge\ldots\wedge   s_k
\\
&=
\sum_{i<j} \sum_{e\in L_i\times L_j}(-1)^{i+j+1} (s_i\cdot s_j) _e\wedge s_1 \wedge\ldots \hat s_i\wedge \ldots \wedge \hat s_j\wedge\ldots\wedge   s_k.
\end{align*}

Let  $e=\{\lambda,\mu\}$ be any (unordered) pair of distinct spider legs.
We may assume $\lambda\in s_i$ and $\mu\in s_j$ with $i\leq j$.  To simplify notation in what follows, we define  $X_e$ to be zero if $i=j$ and otherwise

$$X_e=(-1)^{i+j+1} (s_i\cdot s_j) _e\wedge s_1 \wedge\ldots \hat s_i\wedge \ldots \wedge \hat s_j\wedge\ldots\wedge   s_k.$$
Then
$$
\bdry_{\text{Lie}}(X)=\sum_{e\in E} X_e,
$$
where $E$ consists of all pairs $\{\lambda,\mu\}$ of legs with $\lambda\neq \mu$.

\begin{definition}
The \emph{degree} of a wedge $s_1\wedge\ldots\wedge s_k$ is the sum of the degrees of the $s_i$.
\end{definition}

Since $\bdry_{\text{Lie}}$ preserves degree,
$\ext \lplus_V$ breaks up into a direct sum of chain complexes
$\ext \lplus_V\degree{d}$ according to degree, and we define $H_*(\lplus)\degree{d}$ to be the homology of this chain complex.

Since spiders can have any positive degree, the spaces of $k$-chains $\Wedge^k\lplus_V$ are  generally not finite-dimensional, so the definition of
\emph{cohomology} becomes problematic.  We resolve this using the observation that
if $V$ is finite-dimensional then there are only finitely-many ways to decorate the legs of a basic spider, so
the degree $d$ chains $\Wedge^k\lplus_V\degree{d}$ \emph{are} finite-dimensional.

\begin{definition}
If $V$ is finite-dimensional, then the \emph{continuous cohomology}  $H^*_c(\lplus_V)$
is the cohomology computed from the cochain complex whose $k$-cochains are
$$
C_c^k\lplus = \bigoplus_{d\geq 1} \left(\Wedge^k\lplus_V\degree{d}\right)^*.
$$
\end{definition}

If $V$ is infinite-dimensional, we realize it as the direct limit of finite-dimensional subspaces $V_n$. The inclusion maps $V_n\to V_{n+1}$ dualize to give maps $H^*_c(\lplus_{V_{n+1}})\to H^*_c(\lplus_{V_n})$.
\begin{definition}\label{cont}
For $\displaystyle V=\dirlim V_n$, the continuous cohomology  $H^*_c(\lplus_{V_\infty})$ is the inverse limit of the  maps
$H^*_c(\lplus_{V_{n+1}})\to H^*_c(\lplus_{V_n})$ .\end{definition}

\subsection{Functoriality.} One may regard $\cL=\cLO$ and $\lplus=\lplus\cO$ as
(covariant) functors from the category of symplectic vector spaces to the category of Lie algebras.
In particular,
  any linear map $\phi\colon V \rightarrow W$ which respects the symplectic forms on $V$ and $W$ induces  a Lie algebra homomorphism
$\phi_*\colon \cL_V \rightarrow \cL_W$, where the image of a spider is
obtained by applying $\phi$ to the labels. If $\phi$ is injective (resp.  surjective) then so is $\phi_*$.
The restriction $\phi_*\colon\lplus_V\rightarrow\lplus_W$ to the positive degree subalgebras has the same properties.
Thus, the natural inclusions  $V_n\hookrightarrow V_{n+1}$ induce   inclusions
$\cL_n \hookrightarrow \cL_{n+1}$
and $\lplus_n \hookrightarrow \lplus_{n+1}$.

The functors $\cL$ and $\lplus$ commute with direct limits.
Since
$\lplus_\infty = \dirlim \lplus_n$ and  homology commutes with direct limits,
we have $H_*(\lplus_\infty) = \dirlim H_*(\lplus_{n})$.

The action of the symplectic group $\SP_V$ commutes with the Lie bracket  and hence with the Chevalley-Eilenberg differential, so that $\SP_V$ also acts  on $H_*(\lplus_V)$ and $H_c^*(\lplus_V)$,
 and the association $V\funct H_*(\lplus_V)$ is a functor from symplectic vector spaces to symplectic modules. This functor preserves injections for infinite dimensional spaces:

 \begin{lemma}
\label{lem:injonHlplus}
Let $V\hookrightarrow W$ be an inclusion between two infinite-dimensional
non-degenerate symplectic spaces. Then the induced map
$H_*(\lplus_V)\rightarrow H_*(\lplus_W)$ is injective.
\end{lemma}
\begin{proof}  We identify $V$ with its image in $W$ and $\ext\lplus_V$ with its image in $\ext\lplus_W$.
We need to show that if a cycle $z\in \Wedge \lplus_V$ is a boundary in $\ext \lplus_W$, then is also a boundary in $\ext\lplus_V$.

Fix $x \in \ext \lplus_W$ with $\bdry_{\text{Lie}}(x)=z$.
Since $W$ and $V$ are direct limits we can find finite dimensional symplectic
subspaces $V'\subseteq W'$ of $W$ such that
 $z \in \ext \lplus_{V'}$ and $x \in \ext \lplus_{W'}$.
Since $V$ is infinite dimensional there exists a subspace $V''$ of $V$ containing $V'$ and a symplectic isomorphism
$\pi\colon W' \rightarrow V''$
which is the identity on $V'$.
 Then $\pi_*(x) \in \ext \lplus_{V''} \subset
\ext \lplus_V$ and $\bdry_{\text{Lie}}(\pi_*x) = \pi_*(\bdry_{\text{Lie}}(x)) = \pi_*(z)=z$,
since $\pi_*$ is the identity on $\ext \lplus_{V'}$.
\end{proof}

We remark  that this proof does not work for finite-dimensional spaces $V$ and $W$.

\section{The complex of hairy $\cO$-graphs}\label{complex}
\label{sec:hairy_graphs}
Recall from~\cite{CV} that a vertex $v$ of a graph was said to be
$\cO$-colored if the half-edges incident to $v$ are identified
with the i/o slots  of some element of  $\cO$. An $\cO$-graph was then
defined to be a graph which is oriented in the sense of Kontsevich
and is $\cO$-colored at every vertex.
 These $\cO$-graphs $(G, or)$, modulo the relation
$(G,or)=-(G,-or)$ plus linearity relations in the operad,
span a chain complex which, in each fixed degree $d$, is  isomorphic
to the subcomplex of $\spf_n$-invariants in $\ext\lplus_n\degree{d}$ for $n \gg d$.

We may think of  a (representative of an) $\cO$-graph as something obtained from a wedge of spiders  by taking the ordered disjoint union of these spiders, connecting their legs in pairs by oriented edges, and discarding the leg labels.  In this section we define a more general complex, generated by (equivalence classes of) objects made by connecting only some pairs of legs  by oriented edges and letting the rest of the legs keep their labels.  We will call these \emph{hairy $\cO$-graphs}, i.e. a labeled leg is now to be thought of as a  ``hair."

\subsection{Chains}
Hairy graphs will be based on finite 1-dimensional CW-complexes.  These may have multiple edges or loops, but we do not allow bivalent vertices.  Univalent vertices are called {\em leaves} and vertices which are not univalent are called {\em internal vertices}.  The edge adjacent to a leaf is called a {\em hair}, and the other vertex of a hair must be internal.   Edges which are not hairs are called {\em internal edges}.  An {\em orientation} on such a graph is determined by ordering the internal vertices and orienting the internal edges.  Either switching an edge-orientation or interchanging the order of two internal vertices reverses the orientation of the graph. 

A graph orientation can be defined more formally as follows ({\it cf.} Definition 2 of \cite{CV}).  Denote the set of internal vertices of a graph $G$ by $\mathrm{IV}(G)$, and the set of internal edges by $\mathrm{IE}(G)$. Given an edge $e\in\mathrm{IE}(G)$, there are two half-edges associated to its two ends. Denote this unordered pair by $H(e)$. Finally, given an arbitrary finite-dimensional vector space $W$, define $\det W=\ext^{\dim W} W$. 
\begin{definition} An orientation of a  graph $G$ is a unit vector in the $1$-dimensional vector space $$\det \R \mathrm{IV}(G)\otimes\bigotimes_{e\in \mathrm{IE}(G)}\det\mathbb RH(e).$$
\end{definition}

\begin{definition}
\label{def:hairy_graph}
A \emph{hairy $\cO$-graph} $\G=(G, or, \{o_v\}, \{x_\lambda\})$ is an oriented graph with no bivalent vertices, where every internal vertex $v$ is colored by an operad element $o_v$ and every leaf $\lambda$ is labeled by an element $x_\lambda$ of $V$.

 If all internal vertices are colored by basis elements of $\cO$ and all labels are in $\cB$, the hairy graph is called a \emph{basic} hairy $\cO$-graph.
A hairy $\cO$-graph is called \emph{primitive} if the underlying graph is connected.
\end{definition}

We may also think of a  hairy $\cO$-graph  as something obtained from a wedge of positive-degree $\cO$-spiders   by taking the ordered disjoint union of these spiders, connecting  some of their legs in pairs by oriented edges and discarding those leg labels, but retaining the  labels on unpaired legs.

We define $\hairy_V=\hairy\cO_V$ to be the vector space spanned by hairy $\cO$-graphs  $(G, or,\{o_v\},  \{x_\lambda\})$, modulo orientation relations $(G,or,\{o_v\},  \{x_\lambda\})=-(G,-or,\{o_v\},  \{x_\lambda\})$ plus linearity  relations on the labels $x_\lambda$ and the operad elements $o_v$.
Then $\hairy_V$ is graded by the number of internal vertices:
$$
\hairy_V=\bigoplus_kC_k\hairy_V,
$$
where $C_k\hairy_V$ is  spanned by hairy $\cO$-graphs   formed from $k$ spiders, i.e. whose underlying graph $G$ has $k$ internal vertices.

Finally, we define $\phairy_V$ to be the subspace of $\hairy_V$ generated by
connected hairy $\cO$-graphs.
Note that the correspondences
$V \funct \hairy_V$ and
$V \funct \phairy_V$
are functorial in $V$. As usual, if  $V=V_n$ or $V_{\infty}$
we will denote the corresponding complexes by $\hairy_n$,  $\hairy_\infty$,
$\phairy_n$,  and $\phairy_\infty$.

\subsection{Boundary operator}
We define a boundary operator
$\bdry_\hairy\colon C_k\hairy_V\rightarrow C_{k-1}\hairy_V$,
which sums over all possible ways of using an edge to merge two operad elements.
We need to be careful about the sign here, in order to make $\bdry_{\hairy}^2=0$.

Suppose $e$ is an internal edge of a hairy $\cO$-graph $\G=(G, or,\{o_v\},  \{x_\lambda\})$.
Choose a representative of the orientation on $G$ so that $e$ is oriented
from $v_i$ to $v_j$ with $i\leq j$.
If $i=j$, define $\G_e=0$; otherwise define a new $\cO$-graph $\G_e$
by the following procedure:
\begin{enumerate}
\item collapse the edge $e$ to a single vertex,
\item fuse the operad elements coloring  $v_i$ and   $v_j$
into a single spider using the i/o slot corresponding to
the initial half of $e$ as output and the i/o slot corresponding to
the terminal half of $e$ as input,

\item  number the vertices of the collapsed graph so that
the resulting operad element  tints the first vertex and the remaining vertices
retain their relative order (and their tints),

\item multiply by $(-1)^{i+j+1}$.
\end{enumerate}

Note that the sign convention imitates  the sign in the Chevalley-Eilenberg boundary operator.
The boundary operator $\bdry_\hairy\colon C_k\hairy_V\to C_{k-1}\hairy_V$ is then defined by
$$
\bdry_\hairy(\G) = \sum_{e\in E(G)}  \G_e,
$$
where $E(G)$ is the set of internal edges of $G$.  It is straightforward to check that this is well-defined, and that
$\bdry_\hairy \circ \bdry_\hairy =0$.

\begin{remark}
Note that the boundary operator  makes no use  of the symplectic form on $V$.
For any linear map $\phi:V \rightarrow W$,  the induced map $\phi_*: \hairy_V \rightarrow \hairy_W$ commutes with the
differential $ \bdry_\hairy$ so that there is an induced map
$\phi_*:H_*(\hairy_V) \rightarrow H_*(\hairy_W)$.  Thus $V\funct \hairy_V$  (resp. $V \funct H_*(\hairy_V)$)
are functors from the category of vector spaces
to the category of complexes (resp. graded vector spaces).
\end{remark}

\subsection{Gradings on $\hairy_V$}
\label{def:gradings}
There are several different things to count in a hairy graph, and we will have occasion to use different ones at various points in this paper.
The most important of these is the {\it degree} of a hairy graph $\G$, which is defined to be sum of the degrees of the spiders coloring the vertices.  Recall that we required these degrees to be positive, so that there are no hairy graphs of degree $0.$   The boundary operator $\bdry_\hairy$   preserves degree, so that  $\hairy_V$ breaks up as a direct sum of  subcomplexes $\hairy_V\degree{d}$ according to degree.  Note that if $V$ is finite-dimensional, then $\hairy_V\degree{d}$ is also finite-dimensional, since there are only finitely many graphs of a given degree.

Other things we can count include:
\begin{enumerate}
\item the number of internal edges in the underlying oriented graph $G$ (the boundary operator decreases this by one);
\item  the number of hairs  in $G$ (the boundary operator preserves this);
\item  the rank of  the first homology $H_1(G)$, which we call the \emph{rank} of $\G$ (the boundary operator preserves this).
\end{enumerate}

\subsection{Homology}

Recall that an inclusion $V\hookrightarrow W$ of symplectic vector spaces is only guaranteed to induce an injection
$H_*(\lplus_V)\to H_*(\lplus_W)$ if $V$ and $W$ are infinite-dimensional and non-degenerate (Lemma~\ref{lem:injonHlplus}). An advantage of the hairy graph homology functor is that it ignores the symplectic form, so is better behaved:

\begin{lemma}
\label{lem:injonHhairy}
The functor $V \funct H_*(\hairy_V)$ preserves injections and surjections.
\end{lemma}
\begin{proof}
Let $i\colon V \hookrightarrow W$ be an inclusion of vector spaces. Since
every subspace has a complement we can find a projection $p:W \to V$
such that the composition $p\circ i$ is the identity.
Applying the functor to this map gives
\begin{diagram}[nohug]
V      &                &   &   \hspace{2cm} & H_*(\hairy_V) &            &  \\
\dEq   & \rdInto(2,1)^i & W &                & \dEq          & \rdTo(2,1)^{i_*} & H_*(\hairy_W) \\
V      & \ldTo(2,1)_p   &   &                & H_*(\hairy_V) & \ldTo(2,1)_{p_*} &
\end{diagram}
By functoriality we have that $p_*\circ i_* = (\id)_* = \id$
therefore $i_*$ is injective.

The proof that the functor preserves surjections is similar.
\end{proof}

\begin{remark}
\label{rem:not_ab_category} The reason this proof does not work for  $H_*(\lplus_V)$ is that there is generally no projection $p\colon W\to V$  which preserves the symplectic form, so there is no induced map $H_*(\lplus_W)\to H_*(\lplus_V)$.  In category-theoretic language,  not every object is projective (or injective) in the category of symplectic vector spaces.
\end{remark}

\section{The Trace map}
\label{sec:Trace_map}

There is an obvious inclusion $\iota\colon\ext\lplus_V\rightarrow \hairy_V$
which sends a wedge $X$ to an $\cO$-graph with no oriented edges,
i.e. it simply erases the wedge symbols, keeping the ordering of the spiders.
However, this is not a chain map:  the differential $\bdry_\hairy$
is zero on the image, but the Chevalley-Eilenberg differential on $\ext \lplus_V$ is  not zero.
In this section we define a new map $\Tr_V\colon\ext\lplus_V\to \hairy_V$,
which \emph{is} a chain map.  Basically, the map $\Tr_V$ sums over all
possible ways to match spider legs to form an $\cO$-graph.
Unless we need to specify the vector space $V$,
we will denote the trace map simply by $\Tr$. We remark  that this trace is not the trace considered in \cite{CVMorita} but is closely related.

\begin{definition}
\label{def:mathicing}
Let ${\G}=(G, or,\{o_v\},  \{x_\lambda\})$ be a hairy $\cO$-graph, with hairs $L=L(\G)$.  A {\em matching} $M$  is a partition of a subset of $L$ into
unordered pairs.
\end{definition}

\noindent Given a matching $M$, let  $\G^M$  denote the  element of $\hairy$  obtained by
\begin{enumerate}
\item  connecting each pair of hairs $\{\lambda,\mu\}\in M$   with an oriented edge from $\lambda$ to $\mu$,
\item  erasing the labels $x_\lambda$ and $x_\mu$,  and
\item  multiplying the result by the product    $\prod \omega( x_\lambda,x_\mu)$.
\end{enumerate}
Note that $\G^M$ is a well-defined element of $\hairy_V$.  In particular, it doesn't matter whether we orient an edge from  $\lambda$ to $\mu$ and multiply by $\omega( x_\lambda,x_\mu)$ or orient it from $\mu$ to $\lambda$ and multiply by $\omega( x_\mu,x_\lambda)$.

For $X=s_1\wedge\ldots\wedge s_k$,  we define
$$
\Tr(X) = \sum_{M} (\iota X)^M,
$$
where $M$ runs over all possible matchings of the hairs of $\iota X$ (i.e. matchings of the legs of the spiders $s_i$).
The map $\Tr$ can also be described as
$$
\Tr(X)= \exp (T)(\iota X)=  \iota X + T(\iota X ) + \frac{1}{2!}T^2(\iota X )+
    \frac{1}{3!}T^3(\iota X )+\ldots,
$$
where $T\colon\hairy_V\to\hairy_V$ sums over matchings with exactly one pair, i.e.
$$  T(\G)=\sum_{|M|=1} \G^M$$ (see Figure~\ref{T}).
\begin{figure}
\begin{center}
\ifpdf\includegraphics[width=4in]{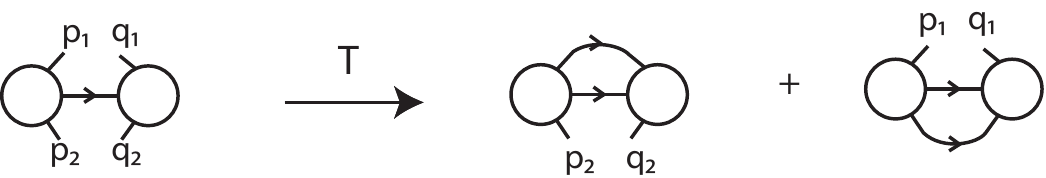}\fi
\caption{The map $T\colon\hairy_\infty\to\hairy_\infty$}\label{T}
\end{center}
\end{figure}  Note that both $T$ and $\exp (T)$ preserve degree, and that neither $T$ nor  $\exp (T)$ is a chain map.

\begin{proposition}
\label{prop:Tr_as_chain_map}
$\Tr$ is a chain map.
\end{proposition}

\begin{proof}
We need to show $\bdry_\hairy\circ \Tr = \Tr\circ\bdry_{\text{Lie}}$.  A pair of spider-legs $\{\lambda,\mu\}$ is called \emph{internal} if they are on the same spider, and \emph{external} if they are on different spiders.  Applying $\Tr$ sums over all possible matchings; then applying $\bdry_\hairy$ sums over all ways of fusing a single external pair from each matching.  On the other hand $\bdry_{\text{Lie}}$ sums over all ways of matching one external pair  \emph{and} fusing it; applying $\Tr$ then gives all possible ways of matching subsets of the remaining legs.  In either case, the result is the sum of all matchings with one external pair fused.  The signs in the definition of $\bdry_\hairy$ are designed to agree with  $\bdry_{\text{Lie}}$.
\end{proof}

\bigskip
The main reason for studying the hairy graph complex $\hairy_\infty$ is to be able to get information about $H_*(\lplus_\infty)$ through knowledge of $H_*(\hairy_\infty)$. The following theorem, in combination with Theorem~\ref{thm:surjective}, allows us to do this. In the sequel to this paper we will analyze the case of finite $n$, showing that for large $n$, $\Tr_n$ is ``almost" injective.

\begin{theorem}
\label{th:Tr_injection}
\label{Th:injective}
\label{injection}
$\Tr_\infty$ induces an injection
$H_*(\lplus_\infty)\hookrightarrow H_*(\hairy_\infty)$.
\end{theorem}

We would like to prove the theorem by  defining   chain maps $\beta_n\colon\hairy_n\rightarrow\ext\lplus_n$
and proving that $\beta_n\circ \Tr_n(X)=X$, so that $\beta_n\circ \Tr_n$ induces
the identity on homology for all $n$.
The map   $\exp (-T)\colon\hairy\rightarrow\hairy$  is an inverse to $\exp (T)$,
so we just need   a map $\alpha$ which breaks up a hairy $\cO$-graph into
$\cO$-spiders in such a way that $\beta_n=\alpha\circ \exp (-T)$ is a chain map.
This almost works, but not quite;
it turns out that  $\Tr_n$ may  not induce an injection
$H_*(\lplus_n)\rightarrow H_*(\hairy_n)$ for finite $n$, but only becomes injective after stabilization.  To prove injectivity we must therefore be slightly devious when defining $\alpha$.

Any basic hairy $\cO$-graph $\G$ is equal to $(\iota X)^M$ for some wedge $X$ of basic spiders and some matching $M$.  Then $\exp (-T)(\G)$ is  given by the formula
$$
\exp (-T)(\G)=\sum_{M'\supseteq M}(-1)^{m'-m}(\iota X)^{M'}.
$$
where  the sum is over all matchings $M'$ containing $M$  (or,  equivalently, the sum over all matchings of   the hairs of $\G$), $m$ is the size of $M$ and $m'$ is the size of $M'$.

We want a map $\alpha$ which cuts apart each
hairy $\cO$-graph into a wedge of $\cO$-spiders, labels the cut edges with
new labels which are distinct from the old labels, and multiplies the result by a coefficient which will ensure that $\alpha$ is well defined and that  the composition $\alpha\circ \exp (-T)$ is  a chain map.
We will define $\alpha$ separately on the degree $d$ subcomplexes  $\hairy_V\degree{d}$, writing  $\alpha_{d,V}$ for the restriction of $\alpha$ to $\hairy_V\degree{d}$.

A hairy $\cO$-graph $\G$ of degree $d$ has at most $3d$  hairs, so any matching of its hairs has at most $[\frac{3d}{2}]$ elements.   Set $N=[\frac{3d}{2}]$, and
fix a new  $2N$-dimensional vector space $W_d$ with symplectic basis $\cB'= \{p'_{1},q'_{1},\ldots,p'_{N},q'_{N}\}.$ The target of $\alpha_{d,V}$ will be the exterior product $\ext\lplus_{{V}\oplus W_d}\degree{d}$ instead of $\ext\lplus_V\degree{d}$, i.e. the labels on spider legs are allowed to be in $V\oplus W_d$.

For $\G=(\iota X)^M\in\hairy_V\degree{d}$,  define a \emph{state} $s$ of $M$ to be an assignment of  new labels from $\cB'$  to the hairs in $M$ such that
\begin{enumerate}
\item the labels on each pair in $M$ are dual, i.e. they are $ p_i'$ and $q_i'$ for some index $i $ with $1\leq i\leq N$,
\item the indices for  different pairs in $M$ are distinct,
\item replacing the original labels on spider legs in $M$ by the new labels in $s$ results in a new wedge of spiders denoted $X\{s\}$, with $(\iota X\{s\})^M=\pm\G=\pm(\iota X)^M$.
\end{enumerate}
Let $S(M)$ denote the set of all   possible states of $M$.  The number of such states is $$|S(M)|=2N(2N-2)(2N-4)\ldots (2N-2m+2)= {2^mN!}/{(N-m)!}.$$

We define the \emph{sign} of a state   to be $\sigma(s)=1$ if $(\iota X\{s\})^M= \G$ and $\sigma(x)=-1$ if $(\iota X\{s\})^M=-\G$.  We now define $\alpha_{d,V}\colon \hairy_V\degree{d}\to \ext \lplus_{{V}\oplus W_d}\degree{d}$ by summing over all possible states, normalized by the number of possible states:
$$
\alpha_{d,V}(\iota X^M)=\frac{1}{|\mathcal S(M)|} \sum_{s\in \mathcal S(M)} \sigma(s)X\{s\}=\frac{(N-m)!}{2^mN!} \sum_{s\in \mathcal S(M)}  \sigma(s)X\{s\}.
$$

\begin{lemma}
$\beta_{d,V}=\alpha_{d,V}\circ\ \exp (-T): \hairy_V\degree{d} \to \ext\lplus_{V\oplus W}\degree{d}$ is a chain map.
\end{lemma}

\begin{proof}  To   simplify notation we will identify $\ext\lplus_V$ with its image (under $\iota$) in $\hairy_V$ so that $(\iota X)^M$ becomes $X^M$.

We first compute $ \beta_{d,V} \bdry_\hairy(X^M)$:
\begin{align*}
\beta_{d,V}\bdry_{\hairy}(X^M)&
   =\beta _{d,V}\sum_{e\in M}X^{M}_e\\
&=\alpha_{d,V} \sum_{e\in M}\sum_{M'\supseteq M}(-1)^{(m'-m)}X^{M'}_e\\
&=\sum_{e\in M}\sum_{M'\supseteq M}(-1)^{(m'-m)}\frac{1}{|S(M'- e)|}\sum_{s\in S(M'- e)}\sigma(s)X_e\{s\}.\\
\end{align*}
On the other hand,
\begin{align*}
\bdry_{\text{Lie}}\beta_{d,V}(X^M)&
=\bdry_{\text{Lie}}\alpha_{d,V} \sum_{M'\supseteq M}(-1)^{(m'-m)}X^{M'}\\
&=\bdry_{\text{Lie}}\sum_{M'\supseteq M} (-1)^{(m'-m)}\frac{1}{|S(M')|}\sum_{s\in S(M')} \sigma(s)X\{s\} \\
&=\sum_{M'\supseteq M}(-1)^{(m'-m)}\frac{1}{|S(M')|}\sum_{s\in S(M')}\sum_{e\in E}\sigma(s) X\{s\}_e\\
&=\sum_{e\in E}\sum_{M'\supseteq M}(-1)^{(m'-m)}\frac{1}{|S(M')|}\sum_{s\in S(M')}\sigma(s) X\{s\}_e,
\end{align*}
where $E$ is the set of pairs of distinct legs of $X$.

The lemma now follows from the following two observations:

\begin{enumerate}
\item If $e\in M$, then $\frac{1}{|S(M')|}\sum_{s\in S(M')} \sigma(s)X\{s\}_e=\frac{1}{|S(M'- e)|}\sum_{s\in S(M'- e)}\sigma(s)X_e\{s\}$.

\item If $e\not\in M$, then the terms $(-1)^{(m'-m)}\frac{1}{|S(M')|}\sum_{s\in S(M')} \sigma(s)X\{s\}_e$ cancel.

\end{enumerate}
Both observations are true because each time the term $\sigma(s)X\{s\}_e$ arises from a matching $M'$ containing $e$ it occurs $2(N -(m' - 1))$ times, but it also arises once from the matching $M' - {e}$ with opposite sign.
\end{proof}

Returning to the functorial perspective, we have the following lemma.
\begin{lemma}
\label{lem:Tr_as_natural_transformation}
The family of maps $\Tr_V$ and $\beta_{d,V}$ form  natural transformations between the functors
$\ext\lplus_\dash\degree{d}\funct\hairy_\dash\degree{d}$ and
$\hairy_\dash\degree{d} \funct \ext\lplus_{\dash\,\oplus W_d}\degree{d}$ respectively.
\end{lemma}
\begin{proof}
Let $\phi:V \to V'$ be a linear map which preserves the symplectic form.
The naturality of $\Tr$ and $\beta$ is equivalent to commutativity of
the diagrams
\begin{diagram}
\ext\lplus_V\degree{d} & \rTo^{\phi_*} & \ext\lplus_{V'}\degree{d}
& &
\hairy_V\degree{d} &  \rTo^{\phi_*} &  \hairy_{V'}\degree{d}
\\
\dTo_{{\Tr_V}\!\!\!\!\!\!\!\!\!} & \hspace{20mm}& \dTo_{{\Tr_{V'}\!\!\!\!\!\!\!\!\!}}
& \hspace{20mm}  &
\dTo_{\beta_{d,V}\!\!\!\!\!\!\!\!\!} & \hspace{20mm} & \dTo_{\beta_{d,V'}\!\!\!\!\!\!\!\!\!}
\\
\hairy_{V}\degree{d} &  \rTo^{\phi_*} &  \hairy_{V'}\degree{d}
& &
\ext\lplus_{V\oplus W_d}\degree{d} & \rTo^{\phi_*} & \ext\lplus_{V'\oplus W_d}\degree{d}
\end{diagram}
i.e. $\phi_* \circ \Tr_V = \Tr_{V'} \circ \phi_*$ and
$\phi_* \circ \beta_{d,V} = \beta_{d,V'} \circ \phi_*$.
The commutativity of these diagram follows immediately from the definitions of the maps $\Tr$ and $\beta$.
\end{proof}

\begin{proof}[Proof of Theorem~\ref{th:Tr_injection}.]
For any $V$ we have
$$
\beta_{d,V}(\Tr (X)) = \alpha_{d,V}(\exp (-T)(\exp (T)(\iota X)))) =\alpha_{d,V} (\iota X).
$$
But   $\iota X$ is an $\cO$-graph with no oriented edges,
so $\alpha$ has no edges to break and we just have $\alpha_{d,V}(\iota X)=X$.  In particular,  for $V=V_\infty$ we have $\beta_{d,\infty}\Tr_\infty$ is just the inclusion
 $\ext\lplus_\infty \degree{d}\hookrightarrow \ext\lplus_{{V_\infty\oplus W_d}}\degree{d}$.
By Lemma~\ref{lem:injonHlplus},  this inclusion induces an injection
on homology, which implies that
$\Tr_\infty\colon H_*(\ext\lplus_\infty\degree{d}) \to H_*(\hairy_\infty\degree{d})$ is injective.
Since the complexes $\ext\lplus_\infty$ and $\hairy_\infty$
break up as direct sums according to the degree $d$, we have that
$\Tr_\infty:\ext\lplus_\infty \to \hairy_\infty$ is injective on homology.
\end{proof}

Since $\Tr_*$ is injective, computing $H_*(\hairy_\infty)$  gives us a sort of  ``upper bound" on  $H_*(\lplus_\infty)$.  The next theorem gives a complementary ``lower bound."

 Let $V_n^+\subset V_n$ be the subspace spanned by  $\{p_1,\ldots,p_n\},$  and $V_\infty^+\subset V_\infty$ the union of all the $V^+_n$.  Let $\hairy^+_\infty\subset \hairy_\infty$ be the subcomplex spanned by hairy graphs with all labels in $V_\infty^+$.  The projection $p\colon \hairy_\infty\to \hairy^+_\infty$ sending a hairy $\cO$-graph $\G$ to itself if all labels are in  $V_\infty^+$  and to zero otherwise is clearly a surjective chain map.

\begin{theorem}\label{thm:surjective}
The composition
\begin{diagram}[inline]
H_*(\lplus_{\infty})
& \rTo^{\Tr_*}  &
H_*(\hairy_{\infty})
& \rTo^{p_*} &
H_*(\hairy_{\infty}^+)
\end{diagram}
is surjective.
\end{theorem}

\begin{proof}  Since the complexes $\ext\lplus_\infty$ and $\hairy_\infty$
break up as direct sums according to degree $d$, it suffices to show this separately for each $d$.

Any cycle $z$ in $ \hairy_\infty^+\degree{d}$ is a sum of hairy graphs $\G$, all of whose labels are in $V_n^+\subset V_\infty^+$ for some $n$.  Then $z$ is also a cycle in $\hairy_n\degree{d}$, and since $\beta_{d,n}$ is a chain map $\beta_{d,n}(z)$ is a cycle in $\ext\lplus_{n\oplus W_d}\degree{d}$.  Choose an isomorphism of $W_d$ with a subspace of $V_\infty$ which sends each $\{p_i', q_i'\}$ to  $\{p_j,q_j\}$ for some $j>n$, and apply this isomorphism to the labels of  $\beta_{d,n}(z)$ which are in $W_d$.  The result  is still a cycle $\tilde\beta_{d,n}(z)$ but is now in $\ext\lplus_\infty$.   We claim that the image of this cycle under $\Tr\circ p$ is equal to  $z$.

Since the symplectic product is zero on the labels of $z$, $\exp (-T)(z)=z$, so $\beta_{d,n}(z)=\alpha_{d,n}(z)$.  Now  $\alpha_{d,n}(z)$ is obtained by breaking all edges of graphs in $z$ and labeling the resulting legs by labels in $W_d$.  In $\tilde\beta_{d,n}(z)$ the  labels in $W_d$ are replaced by labels in $V_\infty$.  We now apply $\Tr=\exp (T)\iota$ to $\tilde\beta_{d,n}(z)$.  The only matchings of $\iota\tilde\beta_{d,n}(z)$ which give non-zero terms when we apply $\exp (T)$ are those which originally came from edges of $z$.  The projection  $p$ then kills all terms of $\Tr(\tilde\beta_{d,n}(z))$ except the term which rematches all of the edges of $z$.  In other words, $p(\Tr(\tilde\beta_{d,n}(z))=z$.    Since every cycle is in the image of $p\circ \Tr$, the induced map on homology is surjective.
\end{proof}

Since  $H_*(\hairy^+_\infty)$ generates $H_*(\hairy_\infty)$ as a  $\GL(V_\infty)$-module,
Theorem~\ref{thm:surjective} shows that the image of $\Tr_*$ is at least large.
In the sequel to this paper, we will show that the $\SP$-module decomposition of the image of  $\Tr_*$ corresponds exactly to the $\GL$-decomposition of $\hairy_\infty$.

\section{Schur functors and the image of $\Tr_*$ }
We have defined several functors from the   category of vector spaces to itself.
By  classical representation theory  any such functor can be decomposed as a
direct sum of  \emph{Schur functors\,}  $\SF{\lambda}$  indexed by partitions $\lambda$.
The $\SF{\lambda}V$ are called \emph{Weyl modules}; they are nontrivial irreducible
representations of $\GL(V)$ if the dimension of $V$ is sufficiently large.
The module $\SF{\lambda}V$ can be defined using the irreducible representation $P_\lambda$ of the symmetric group $\Sigma_n$ corresponding to the partition $\lambda$ via $$
\SF{\lambda}V = P_\lambda \otimes_{\Sigma_n} V^{\otimes n},
$$
where the symmetric group acts on the tensor power $V^{\otimes n}$ by permuting the factors.
If $\lambda=(k)$ is the trivial partition of $k$, then $\SF{\lambda}V$ is the $k$-th symmetric power $S^kV,$ and if  $\lambda=(1,1,\ldots,1)$  then $\SF{\lambda}V$ is the $k$-th exterior power $\ext^kV.$  If $\dim(V)=n$ and $\lambda=(m,k-m)$, then
$$
\dim \SF{(m,k-m)}(V)=\frac{2m-k+1}{m+1}\binom{n-2+(k-m)}{k-m}\binom{n+m-1}{m}.
$$
In general if $\lambda$ is a partition of $k$, then   $\SF{\lambda}V$ is the image of the action of the   Young symmetrizer  $c_\lambda\in \F[\Sigma_k]$ on  $V^{\otimes k}$  (see, e.g.~\cite{FH}).

If the vector space $V$ has  a symplectic structure then $\SF{\lambda}V$ is also a representation of  $\SP(V)$, but  is not necessarily an irreducible representation.
It does have  a large irreducible component denoted by $\SpF{\lambda}V$.  If $V^+$ is any Lagrangian subspace of $V,$ then $\SpF{\lambda}V$  is   generated as an $\SP(V)$-module by $\SF{\lambda}V^+$, provided that the dimension of $V$ is large.

Applying this to the functor giving the degree $d$ component of $H_*(\hairy_V)$, we decompose
$$
H_*(\hairy_V\degree{d}) = \bigoplus {(\SF{\lambda}V})^{\oplus m_{d,\lambda}}
$$
and
$$
\SP(V) \cdot H_*(\hairy_{V^+}\degree{d}) = \bigoplus {(\SpF{\lambda}V)}^{\oplus m_{d,\lambda}}
$$

We will show in \cite{CKV2} that in fact $\SP(V) \cdot H_*(\hairy_{V^+}\degree{d})$ coincides with $
 \im \Tr_*$.

\section{$H_1(\hairy)$ and $\lplus^{\text{ab}}$ for the commutative operad}

Recall that the commutative operad   $\Com((n))=\F$ for all $n\geq 2$, with all compositions induced by multiplication in $\F$.  In this case all of the constructions we have given are very simple, especially in dimension 1.  We go through them here as a warm-up exercise.

\subsection{The commutative Lie algebra}
Basic commutative spiders can be thought of as star graphs, i.e. connected graphs with one central vertex and all other vertices univalent, labeled by basis elements of $V$.  Two spiders are fused by identifying a leg of one spider with a leg of the other, then collapsing this leg  and multiplying the result by the symplectic product of the leg labels.
In terms of Schur functors, for all $d$,
the functor  $V\funct  {\mathcal {LC}om_V\degree{d}}$   is simply
$$
\mathcal{L C}om_V\degree{d}\cong \SF{(d+2)}V\cong S^{d+2}V,
$$
where this is an isomorphism of $\GL(V)$-modules.

\subsection{Commutative hairy graph homology in dimension 1}
A hairy graph is just a finite graph with no bivalent vertices, whose univalent vertices are labeled by elements of $V$.  Since there are no hairy graphs with $0$ vertices $C_0\hairy=0$ and the first homology of $\hairy$ is the quotient of $C_1\hairy$ by the image of $\bdry_\hairy\colon C_2\hairy\to C_1\hairy$.  The generators of $C_1\hairy$ are hairy graphs $\G$ with one vertex.
If this vertex $v$ has valence at least $4$, then the half-edges at $v$
can be partitioned into two sets, each of size at least $2$,
such that no oriented edge has  
its half-edges in different pieces of the partition.
We can use this partition to blow up $v$ into an oriented edge $e$ in a new hairy graph $\G'$; then the boundary map just collapses $e$, and $\bdry_\hairy\G'=\G$.  If $v$ has valence $3$, then $\G$ cannot be in the image of $\bdry_\hairy$ becase $\bdry_\hairy$ preserves degree, and there are no graphs in $C_2\hairy$ of degree 1.  Therefore $H_1(\hairy)$ is generated by tripods and loops with one hair, but the loops with one hair have an orientation-reversing automorphism, so are zero.  The labels on the hairs of the tripod give an isomorphism
$$
H_1(\hairy)\cong H_1(\hairy)\degree{1}\cong
S^3V =\SF{(3)}V.
$$

\subsection{The abelianization of $\lplus$} By the previous paragraph, we need only consider spiders of degree 1.
The trace of a basic degree 1 spider in $\lplus$ is either a tripod (if all symplectic pairings on its labels are 0), a tripod plus a graph with one loop and one hair (if there is one non-zero pairing), or a tripod plus twice a  one-loop, one-hair graph (if there are two non-zero pairings).  Thus the image of the trace map is isomorphic to one copy of $S^3V$:

$$
\lplus^{\text{ab}}\cong\im(\Tr_*)\cong S^3V = \SpF{3}V.
$$

\section{$H_1(\hairy)$ and $\lplus^{\text{ab}}$ for the associative operad}

 For the associative operad, $\Assoc((n))$ is the $\F$-vector space with basis given by all cyclic orders of $\{1,\ldots n\}$. Compositions are induced by amalgamating two cyclic orders consistently into one.

The computations of $H_1(\hairy)$ and $\lplus^{\text{ab}}$ are  more subtle than in the commutative case, but the basic plan is the same:  we  first compute the hairy graph homology $H_1(\hairy)$ then compute the abelianization $\lplus_{\rm ab}$ by determining its image under $\Tr_*$ in $H_1(\hairy)$.  The computation illustrates the power of the trace map $\Tr_*$,  allowing us to show that the abelianization is mostly trivial with relative ease. The abelianization has one piece in degree $1$ and one piece in degree $2$ (computed by Morita~\cite{Morita2}) but vanishes for all higher degrees. We remark that in Morita's paper~\cite{Morita2}, the Lie algebra $\lplus$ is denoted by ${\mathfrak a}^+$.

\subsection{The associative Lie algebra}
Basic associative spiders  are now \emph{planar}  star graphs, i.e. connected planar trees with one central vertex and all other vertices univalent and labeled by elements of $V$.  The planar embedding can be thought of as a cyclic ordering on the edges. Two $\Assoc$-spiders are fused by identifying two univalent vertices,  collapsing the adjacent edges  and then multiplying the result by the symplectic product of the associated labels.
For all $d$ we have $\cL\Assoc_V\degree{d}\cong [V^{\otimes{d+2}}]_{\Z_{d+2}}$, the quotient of $V^{\otimes{d+2}}$ by the cyclic action which permutes the factors.   In terms of Schur functors, for small $d$ this decomposes as

\begin{itemize}
\item $\cL\Assoc_V\degree{0}\cong \SF{(2)}V$
\item $\cL\Assoc_V\degree{1}\cong \SF{(3)}V\oplus\SF{(1,1,1)} V$
\item $\cL\Assoc_V\degree{2}\cong \SF{(4)}V\oplus \SF{(2,2)}V\oplus \SF{(2,1,1)}V$
\item $\cL\Assoc_V\degree{3}\cong \SF{(5)}V\oplus \SF{(3,2)}V\oplus  2\SF{(3,1,1)}V\oplus\SF{(2,2,1)}V\oplus\SF{(1,1,1,1,1)}V$
\end{itemize}

\subsection{Associative hairy graph homology in dimension 1}\label{sec:assoc}

The 1-chains $C_1\hairy$ are  generated by basic hairy graphs $\G$  with one vertex $v$.  The half-edges at $v$ are cyclically ordered, and some of them may be joined in pairs by oriented edges $e$.

\begin{definition}
The central vertex $v$ of $\G$ is a  \emph{planar cut vertex} if the  half-edges adjacent to $v$ can be partitioned into two contiguous sets, each with at least two elements, so that every oriented edge has both of its half-edges in the same piece of the partition.
\end{definition}
For example, if $\G$ is not a tripod and $\G$ has two adjacent hairs at $v$, then $v$ is a planar cut vertex.

The boundary map on $C_1\hairy$ is zero, so all elements are cycles, and
the first homology of $\hairy$ is the quotient of  $C_1\hairy$  by the image of the boundary operator $\bdry_\hairy\colon C_2\hairy \to C_1\hairy.$
We begin with some observations about the image of this map.

\begin{lemma}\label{lemma:cut}
Let $\G$ be a generator of  $C_1\hairy$, with central vertex $v$.   If $v$ is a planar cut vertex, then   $\G$   is in the image of $\bdry_\hairy$
\end{lemma}

\begin{proof}  Since $v$ is   a planar cut vertex  it can be blown up  into a separating edge in a new $\Assoc$-graph $\G'$.  Then $\G=\bdry_\hairy(\G')$ (see Figure~\ref{fig:sep}).
\end{proof}

\begin{figure}
\ifpdf
\includegraphics[width=2in]{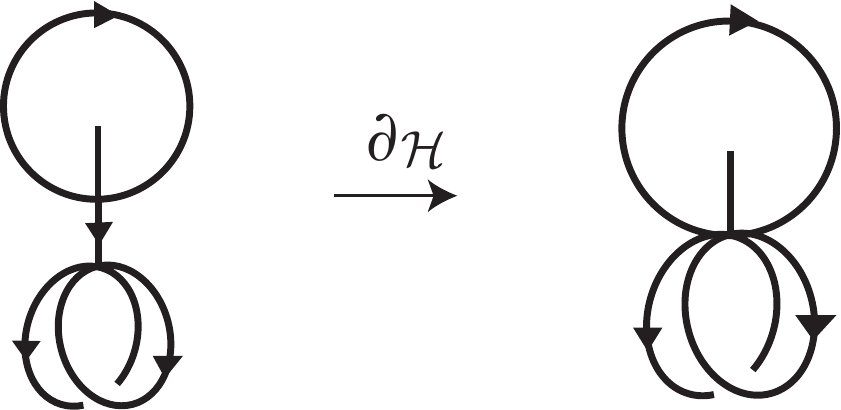}
\fi
\caption{Degree $1$ graphs with planar cut vertices are null-homologous.}\label{fig:sep}
\end{figure}

\begin{lemma}\label{lemma:hairslide}
Let $\G$ be a generator of $C_1\hairy$.  If $\G'$ is obtained from $\G$ by sliding a  hair at one end of an oriented edge of $\G$ along the edge  to the other end, then $\G'$ is homologous to $\G$.
\end{lemma}

\begin{proof}  Midway through the slide we have a 2-vertex $\cO$-graph with two oriented edges between its vertices, whose boundary is  the difference of the original $\cO$-graph and the $\cO$ graph obtained by sliding (see Figure~\ref{fig:assoc}). Thus modulo boundaries, the two $\cO$-graphs are the same.
\end{proof}

\begin{figure}
\ifpdf
\includegraphics[width=4.5in]{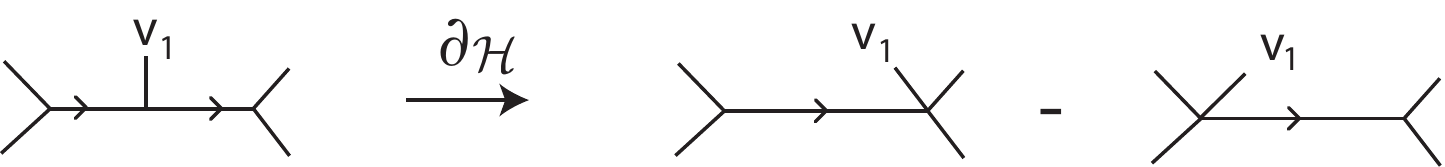}
\fi
\caption{Sliding a hair over an edge}\label{fig:assoc}
\end{figure}

We denote the half-edges of an oriented edge $e$ by $e^-$  and $e^+$.
\begin{definition} Two oriented edges $e$ and $f$  \emph{cross} if $e^-$ and $e^+$ separate $f^-$ and $f^+$  in the cyclic ordering.
\end{definition}

\begin{lemma}\label{lemma:trislide}   Let $\G$ be a generator of $C_1\hairy$, and $a$ an oriented edge of $\G$.  If $a$ is crossed by exactly two oriented edges $e$ and $b$, and if $e$ crosses no other oriented edges,  then $\G$ is homologous to the graph $\G_{e\curvearrowright b}$ obtained by sliding $e$ across $b$.
\end{lemma}

\begin{proof}  Midway through the slide we have a 2-vertex $\cO$-graph $\G'$ with three oriented edges at one of its vertices.  The original graph $\G$ is one term of  $\bdry_\hairy(\G)$.  The second term is the result of the slide.  The third term, obtained by contracting $e$, has a planar cut vertex, so is zero in homology (see Figure~\ref{trislide}).
\end{proof}

\begin{figure}
\ifpdf
\includegraphics[width=5in]{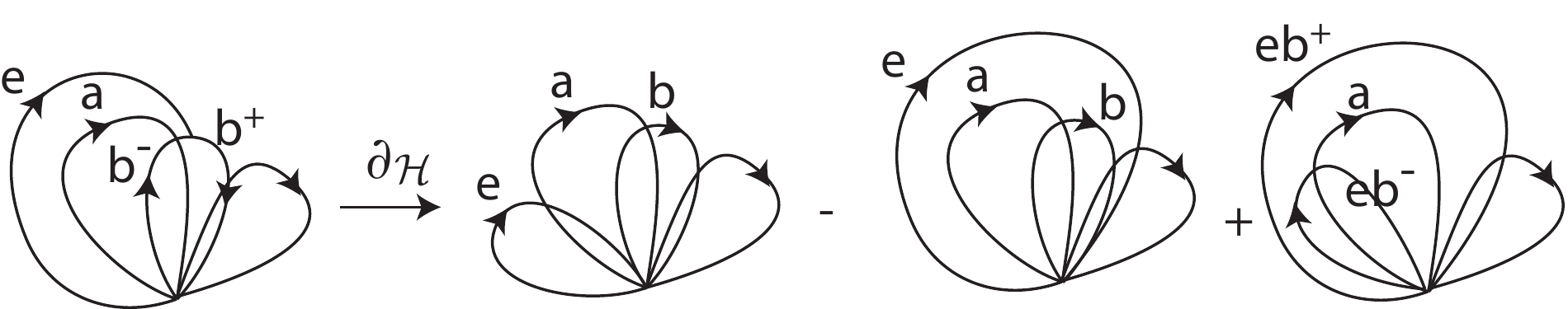}
\fi
\caption{$\bdry_\hairy(\G')=\G + \G_{e\curvearrowright b}+ 0$}\label{trislide}
\end{figure}

\begin{lemma}~\label{lemma:TheMove} Let $\G$ be a generator of $C_1\hairy$, let $e$ be an oriented edge of $\G$, and let $X$ be a set of contiguous half-edges between $e^+$ and $e^-$ which are not all hairs.    Then $\G$ is homologous to a sum of graphs which each have only one half-edge in place of $X$.
\end{lemma}

\begin{proof}  Form a graph $\G'$ by collecting all of the half-edges in $X$ at a single second vertex and then joining the two vertices by an oriented edge $e$.  Then $\G$ is one term of the boundary of $\G'$.  The other terms all have  only one half-edge in place of $X$  (see Figure~\ref{themove}).
\end{proof}

\begin{figure}
\ifpdf
\includegraphics[width=3.5in]{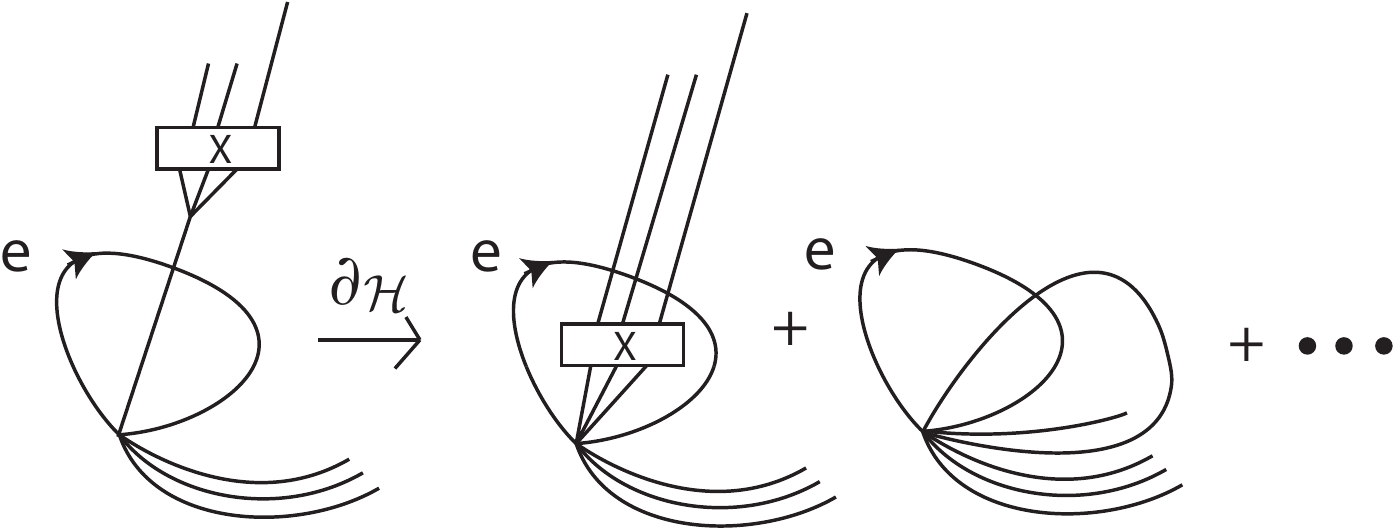}
\fi
\caption{Clearing $X$ from the interior of $e$}\label{themove}
\end{figure}

\begin{theorem} \label{thm:hairyassoc}
For $\cO=\mathcal{A}ssoc$, $H_1(\hairy)\degree{1}$ is generated by tripods and loops with one hair, $H_1(\hairy)\degree{2}$ is generated by  loops with two hairs on opposite sides of the loop, and $H_1(\hairy)\degree{d}=0$ for $d>2$.  As $\GL(V)$-modules, we have
$$
H_1(\hairy)\degree{1}  \cong  \SF{(3)}V \oplus \SF{(1,1,1)}V  \oplus V
\hbox{\rm \, and \, } H_1(\hairy)\degree{2}  \cong  \SF{(1,1)}V.
$$
\end{theorem}

\begin{proof}
This proof is an adaptation of the proof in \cite{mss} to our context.

Let $\G$ be a generator of $C_1\hairy$, i.e. a basic hairy graph with one vertex.   We may assume the central vertex $v$ is not a planar cut vertex, by Lemma~\ref{lemma:cut} .

The cyclic orderings at the vertices of an $\cO$-graph $\G$  give $\G$ a ribbon graph structure, so that $\G$ can be ``fattened" to an oriented  surface with boundary, where we think of the hairs as attached to the boundary.   If two hairs are attached to a single boundary component, then unless $\G$ is a tripod, the hairs  can be slid using Lemma~\ref{lemma:hairslide} to be adjacent, so that $\G$ is a boundary by Lemma~\ref{lemma:cut}.
Thus, we may assume that $\G$ has  at most one hair attached to each boundary component.

If $\G$ has no oriented edges, then $\G$ must be a tripod, by Lemma~\ref{lemma:cut}.

If $\G$ has one oriented edge, then $\G$ is either a loop with one hair or a loop with two hairs, on opposite sides of the loop.   If $\G$ has one hair  it  cannot be the boundary of anything  since the vertex is trivalent, so $\G$ represents a non-trivial element of $H_1(\hairy)$.   If $\G$ has two hairs, then it also represents    a nontrivial homology class. For if $\G=\bdry_\hairy \G'$, then $\G'$ would have to contain a graph which expands the $4$-valent vertex of $\G$ into an edge. There is one such graph up to isomorphism, and it has trivial boundary. Thus $\partial_\hairy \G'\neq \G$. .

If $\G$ has two oriented edges, then since $v$ is not a planar cut vertex the surface must be genus 1 with one boundary component and at most one hair.   The half-edges at $v$ are $e_1^-e_2^-e_1^+e_2^+(h)$, where $h$ is the (possible) hair.   If there is no hair, the automorphism which cyclically permutes these half-edges
$$e_1^- \to e_2^-\to e_1^+\to e_2^+\to e_1^-$$
reverses orientation, so $\G$ is zero in $C_1\hairy$ (see \cite{Jim}, proof of Proposition 2).   If there is a hair, then using Lemma 7.3 the hair can be slid across an edge to produce a homologous graph $G^\prime$ with the opposite orientation, showing that $G=0$ in homology.

Now suppose $\G$ has at least $3$ oriented edges. We may assume they are all oriented in the same direction, say clockwise.  Fix one oriented edge $e_0$  and let $X_0$ be the set of hairs and half-edges between $e_0^-$ and $e_0^+$. Note that $X_0$ cannot consist only of hairs, since then $v$ would be a planar cut vertex.  Using Lemma~\ref{lemma:TheMove}, we see that $\G$ is homologous to a sum of hairy graphs $\G'$, each with only one half-edge  between $e_0^-$ and $e_0^+$; we label this half-edge $e_1^-$.

For each $\G'$, let $X_1$ be the set of hairs and half-edges between $e_1^-$ and $e_1^+$ other than $e_0^+$.  Again note that $X_1$ cannot consist solely of hairs, since then $v$ would be a planar cut vertex.  Applying Lemma~\ref{lemma:TheMove} again, we see that $\G'$ is homologous to a sum of graphs with only $e_0^+$ and one other half-edge, which we label $e_2^-$, between  $e_1^-$ and $e_1^+$.

We continue in this fashion until we run out of oriented edges. The last oriented edge $e_k$ may have a single hair between $e_{k-1}^+$ and $e_k^+$ and/or another after $e_k^+$.   Thus  our original graph $\G$ is homologous to a sum of graphs $\G'$, each with oriented edges   $e_0,\ldots,e_k$; the cyclic ordering on the half-edges    is $e^-_0e^1_-e_0^+e_2^-e_1^+\ldots e_{k}^-e_{k-1}^+(h_1)e_k^+(h_2),$ where the $h_i$ are possible hairs (See Figure~\ref{standard} for the case with no hairs.)

\begin{figure}
\ifpdf
\includegraphics[width=2in]{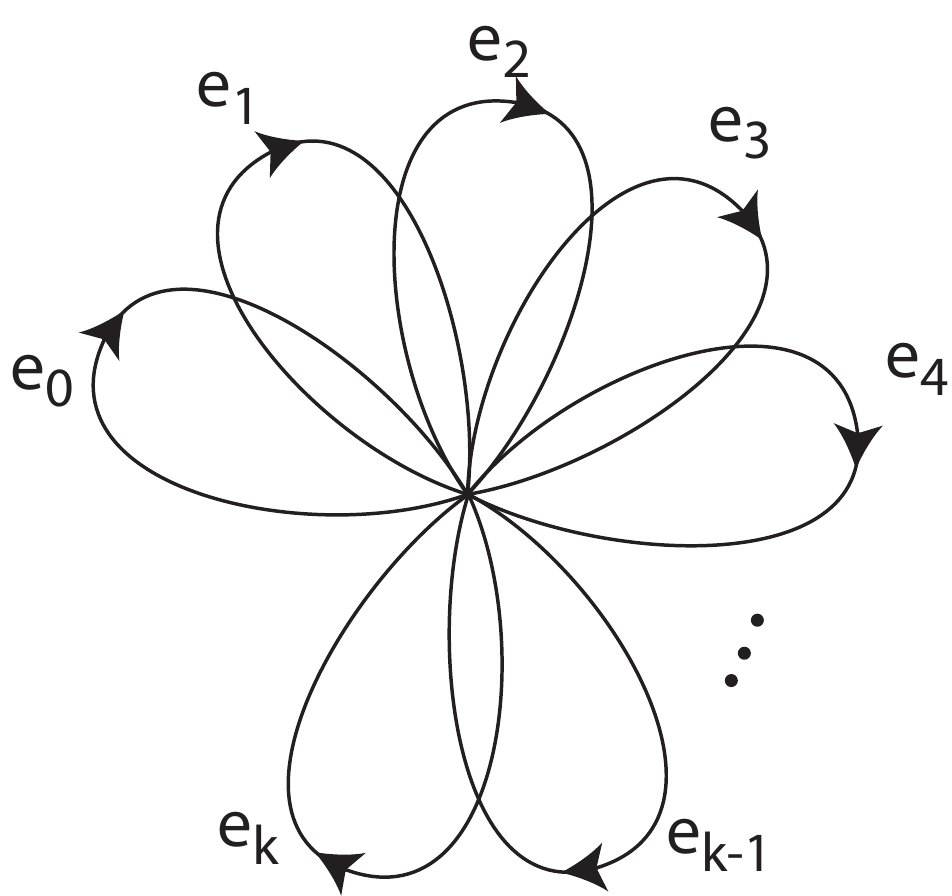}
\fi
\caption{Standard form for generator of $C_1\hairy$.}\label{standard}
\end{figure}

Using Lemma~\ref{lemma:trislide}, in each $\G'$ we may now slide $e_0^+$ successively over $e_2, e_4, \ldots, e_{2[k/2]}$ to obtain a homologous graph $\G''$.  If $k$ is even, then in $\G''$  the two half-edges of $e_0$ are adjacent (possibly with a hair between), i.e. the vertex of $\G''$ is a planar cut vertex so $\G''$ is a boundary.  If $k+1$ is odd, then (after sliding the one possible hair around) $\G''$ is isomorphic to $\G'$ with the orientation on  $e_k$ reversed , i.e. $\G'$ is homologous to  $-\G'$ so is zero.

The non-trivial classes which are represented by loops with $2$ hairs form a copy of $\Wedge^2 V$. The antisymmetry comes from the involution which rotates the $4$-valent vertex to exchange the hairs, switching
the edge orientation in the process.  The classes represented by tripods correspond naturally to   $ V^{\otimes 3} $ modulo the cyclic action of  $\Z_3$.  The graph with one loop and one hair gives a copy of $V$. Thus we have shown $H_1(\hairy)\cong (V^{\otimes 3})_{\Z_3}\oplus V\oplus \Wedge^2 V$, with the first two summands in degree $1$ and the second in degree $2$.
\end{proof}

 \begin{remark} Let $\mathcal H_{(g,s)}$ be the part of the hairy graph complex spanned by hairy graphs that thicken to a surface of genus $g$ with $s\geq 1$ boundary components. Let $\hairy^0_{(g,s)}\subset \mathcal H_{(g,s)}$ be the subcomplex of graphs without hairs.   This is the  standard  ribbon graph complex which computes the  cohomology of $\Mod(g,s)$; in particular, $H_1(\mathcal H_{(g,s)}^0)\cong H^{4g+2s-5}(\Mod(g,s);\F)$ (see \cite{CV}, Section 4). Since graphs in $\hairy^0_{(g,s)}$ have degree $d=4g+2s-4$,  Theorem~\ref{thm:hairyassoc} gives a purely graph homological proof that $H^{*}(\Mod(g,s);\F)$ vanishes in its vcd for $s
\geq 1,2g+s-2> 1.$ In fact, for this same range, one can show using the techniques of section ~\ref{sec:lieoperad} that $H_1(\mathcal H_{(g,s)})\cong H^{4g+2s-5}(\Mod(g,s);(\F\oplus V)^{\otimes s})$, where the $\Mod(g,s)$ action on $(\F\oplus V)^{\otimes s}$ factors through the epimorphism onto the symmetric group $\Mod(g,s)\twoheadrightarrow \Sigma_s$. Theorem~\ref{thm:hairyassoc} therefore also implies the vanishing of the cohomology of $\Mod(g,s)$ in its vcd with coefficients in any   $\Sigma_s$-module.
  \end{remark}

\subsection{The abelianization of $\lplus$}In~\cite{mss}, Morita, Sakasai and Suzuki compute  the   abelianization of $\lplus$.  In this section we point out how this follows from the computation of $H_1(\hairy)$ and injectivity of the trace map.
\begin{theorem}\cite{mss} \label{thm:associative}
For $\cO=\mathcal Assoc$,
$$
\lplus^{\text{ab}}\cong [V^{\otimes 3}]_{\Z_3}\oplus \left(\ext^2 V\right)/\F(\omega_0)
=\SpF{3}V \oplus \SpF{1,1,1}V \oplus \SpF{1,1}V\oplus \SpF{1}V,
$$
where  $\omega_0=\sum_ip_i\wedge q_i$.
\end{theorem}
\begin{proof}
The map $\Tr_*$ is injective, so to determine $\lplus^{\text{ab}}$ it suffices to calculate the image  $\Tr_*(\lplus^{\text{ab}})$ in $H_1(\hairy)$.
The trace map preserves degree, so we do this for degree $1$ and $2$ separately.

The degree $1$ part of $\lplus$ consists of spiders with three legs, and all of these represent nontrivial elements of the abelianization, since everything in the image of the bracket has at least four legs. Thus the degree $1$ part of $\lplus_{\rm ab}$ is isomorphic to $[V^{\otimes 3}]_{\Z_3}$.

The degree $2$ part of $\lplus$ contains 4-legged spiders with labels $a,p_N,b,q_N$  arranged cyclically, where $\omega(p_N,q_N)=1$  is the only non-zero pairing. On the level of homology, the trace of  such a  spider is equal to the loop with two hairs (representing $a\wedge b$), since the other term (the spider by itself) is null-homologous. Thus the entire kernel of  $\omega\colon\ext^2V\to \F$ is in the image of $\Tr_*$, and we have
$$
\im(\Tr_*)\supset \ker(\omega)\cong \ext^2V/\F(\omega_0).
$$
The fact that the image cannot be any larger can be argued by hand  or by appealing to Morita's calculation of the degree $2$ piece.
\end{proof}

In \cite{mss} Morita, Sakasai and Suzuki  combined their results with Kontsevich's theorem to prove that the cohomology of mapping class groups vanishes in their virtual cohomological dimension.

\section{$H_1(\hairy)$ for the Lie operad}\label{sec:lieoperad}

The Lie operad has $\Lie((n))$ spanned by planar uni-trivalent trees with $n$ leaves distinctly labeled by $\{0,\ldots, n-1\}$, modulo the Jacobi identity (IHX relation) and antisymmetry relations. Composition is induced by joining two trees at univalent vertices.

\subsection{The Lie Lie algebra} We can represent a basic Lie spider by drawing a planar unitrivalent tree  and labeling its leaves with basis elements $v\in \cB$.   Two Lie spiders are fused by joining a leg of the first spider to a leg of the second and multiplying the result by the symplectic product of the associated leg labels. In terms of Schur functors, for   small values of $d$ we have

\begin{itemize}
\item $\cL\Lie_V\degree{0}\cong \SF{(2)}V$
\item $\cL\Lie_V\degree{1}\cong \SF{(1,1,1)}V$
\item $\cL\Lie_V\degree{2}\cong \SF{(2,2)}V$
\item $\cL\Lie_V\degree{3}\cong \SF{(3,1,1)}V$
\item $\cL\Lie_V\degree{4}\cong \SF{(4,2)}V\oplus\SF{(3,1,1,1)}V\oplus \SF{(2,2,2)}V$
\end{itemize}

\subsection{Hairy Lie graph homology in dimension 1}   A hairy Lie graph is   represented by an ordered disjoint union of Lie spiders, with some leaves  unlabeled and joined by oriented edges. See Figure~\ref{Liefig}.

In \cite{CV} Section 3.1, it was shown that the Lie graph complex is isomorphic to the ``forested graph complex" which has significantly simpler orientation data. In the presence of hairs, this isomorphism does not quite go through, but one can still simplify the description of a hairy Lie graph slightly.   In the hairy Lie graph of Figure~\ref{Liefig}, this can be done  by removing the grey ovals and noticing that they could be recovered as a neighborhood of the subgraph spanned by all vertices and unoriented edges.
Thus, a hairy Lie graph  may be represented by a uni-trivalent graph  whose univalent vertices are labeled by vectors in $v$ and some of whose internal edges are oriented,  with the property that  the subgraph $\G_u$ spanned by the unoriented edges is a forest containing all of the vertices of $\G$.  Orientation data consists of ordering the components of the forest, and specifying a cyclic ordering of the half edges incident to each trivalent vertex.
The central edge in an IHX relation must be an unoriented edge.

\begin{figure}
\begin{center}
\ifpdf\includegraphics[width=2.5in]{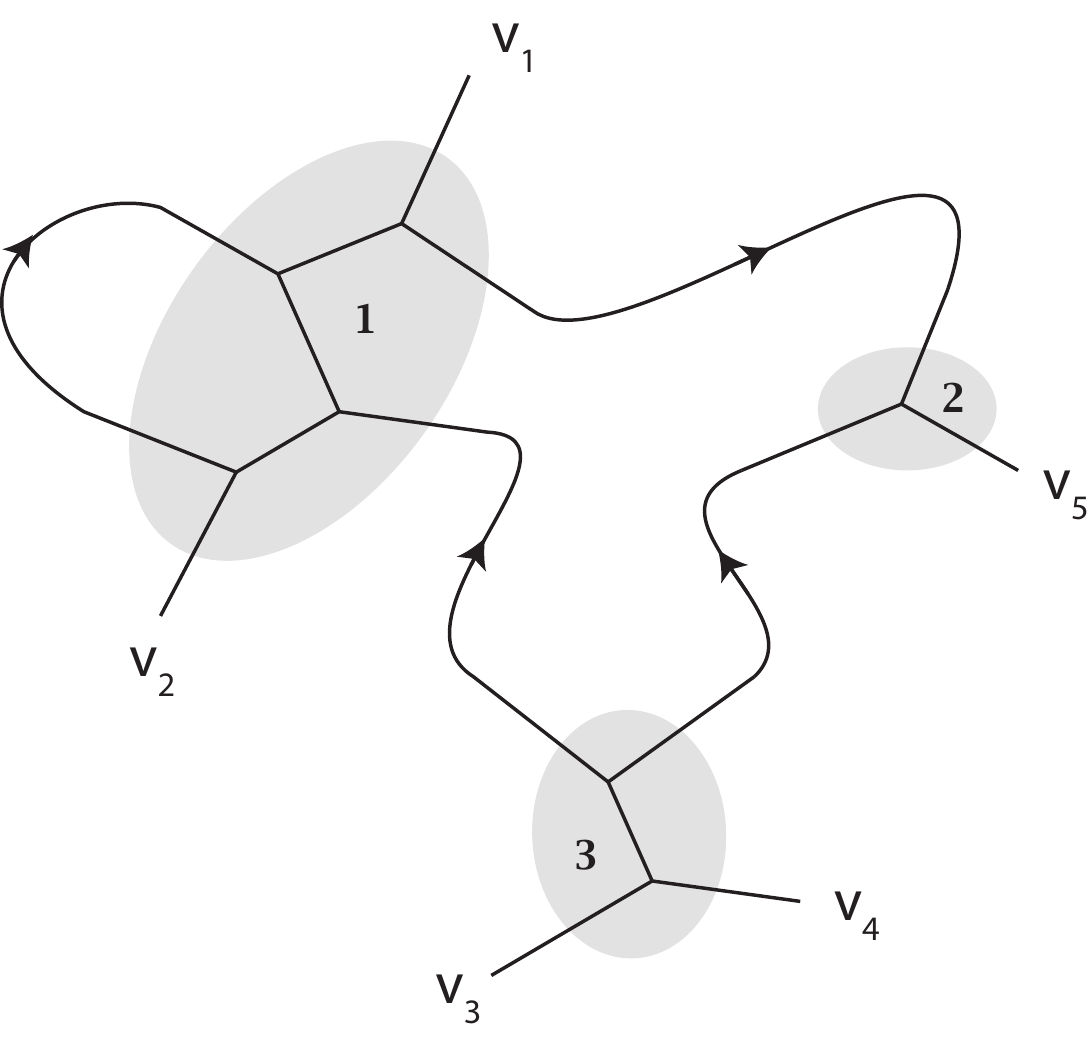}\fi
\caption{A hairy Lie graph with rank $2$ and degree $4+1+2=7$ }\label{Liefig}
\end{center}
\end{figure}

The 1-chains $C_1\hairy$ are  generated by hairy graphs $\G$  such that  the subgraph $\G_u$ spanned by unoriented edges is a (maximal) tree.

The first homology of $\hairy$ is the quotient of  $C_1\hairy$  by the image of the boundary operator $\bdry_\hairy\colon C_2\hairy \to C_1\hairy.$
As in the associative case, we begin with some observations about the image of $\bdry_\hairy$.

\begin{lemma}\label{separatingLie}  Let $\G$ be a Lie graph in $C_1\hairy$.  If some  unoriented edge separates $\G$ into two non-trivial components, then $\G$ is in the image of $\bdry_\hairy$.

\end{lemma}
\begin{proof}   Suppose an edge $e\in \G_u$ separates $\G$ into two non-trivial components, i.e. components which are not a single vertex.    In the Lie graph $\bf H$ obtained by orienting $e$, the unoriented subgraph ${\bf H}_u$ has two components,   and all oriented edges other than $e$ have both ends in one or the other of these components.  Thus, all terms of  $\bdry_\hairy(\bf H)$   are zero except for ${\bf H}_e=\pm\G$.
\end{proof}

In particular,  if  there is an unoriented tree attached to the rest of $\G$ at a single vertex, then $\G$ is in the image of $\bdry_\hairy$ (see Figure~\ref{HangingTree}).

\begin{figure}
$$
\bdry_\hairy\Big(\begin{minipage}{1.0in}
\ifpdf\includegraphics[width=.8in]{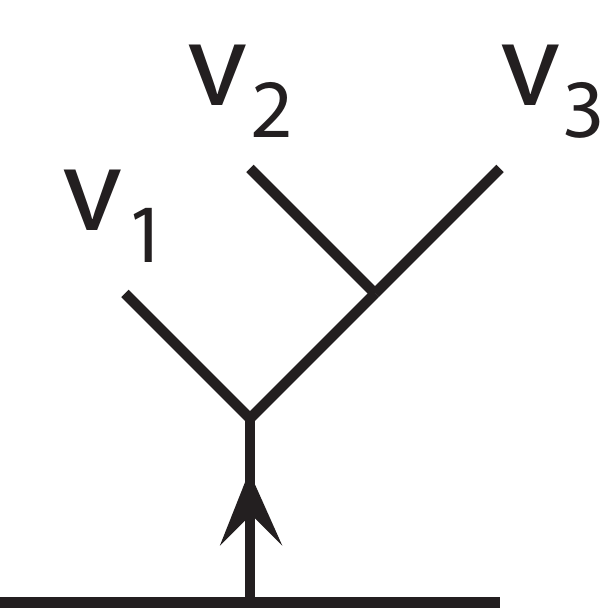}\fi
\end{minipage}\Big)=\begin{minipage}{1.0in}
\ifpdf\includegraphics[width=1.0in]{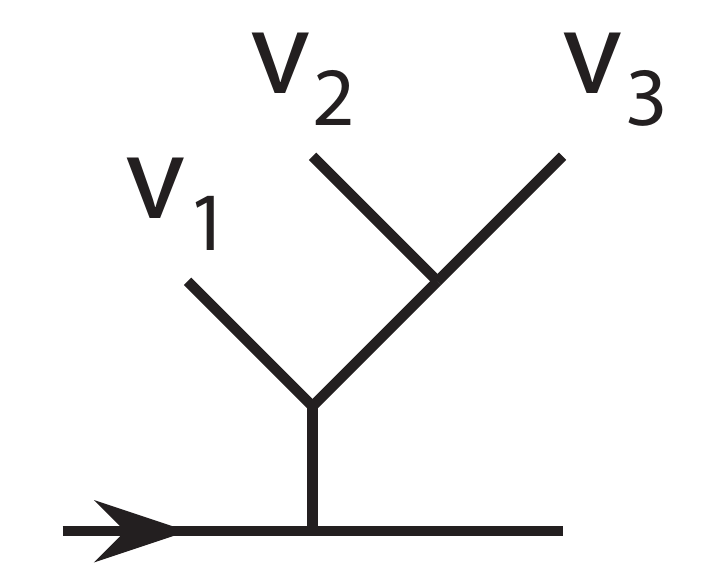}\fi
\end{minipage}
$$
\caption{Attached unoriented tree}\label{HangingTree}
 \end{figure}
Thus, a generator of  $C_1\hairy$  may be thought of as a connected graph $G$ with orientations on the edges in the complement of some maximal tree $T\subset G$,  and with single edges (called \emph{hairs}) attached to some of the edges of $G$.

\begin{lemma}\label{order}
If $\G$ is a Lie graph in $C_1\hairy$, then
hairs  attached to the same unoriented edge of $\G$ may be permuted modulo $\im\bdry_\hairy$.
\end{lemma}

\begin{proof}
This is a consequence of Lemma~\ref{separatingLie} and the   IHX relation
\begin{center}
\ifpdf\includegraphics[width=3.5in]{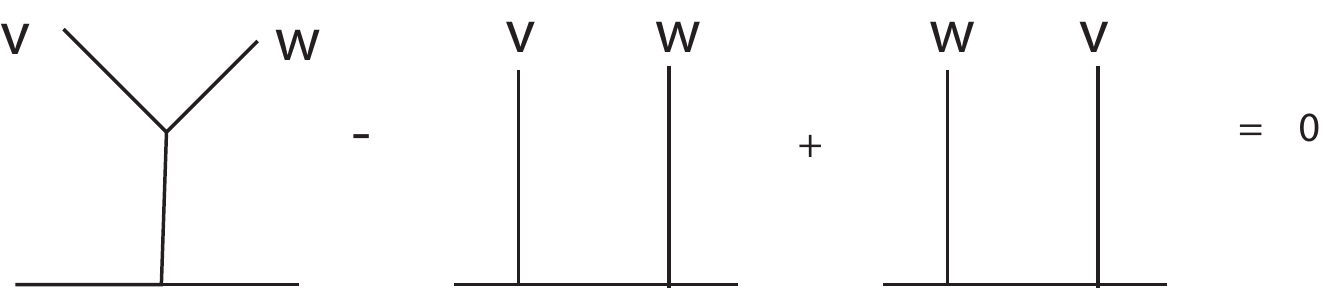}\fi
\end{center}.
\end{proof}

\begin{lemma}\label{slideLie}  If $\G$ is a Lie graph in $C_1\hairy$, then the hairy Lie graph obtained by moving a hair to the other end of
an oriented edge is equal to $\G$ modulo $\im\bdry_\hairy$.
\end{lemma}

\begin{proof}Notice that
$$
\bdry_\hairy\Big(\begin{minipage}{1.0in}
\ifpdf\includegraphics[width=1.0in]{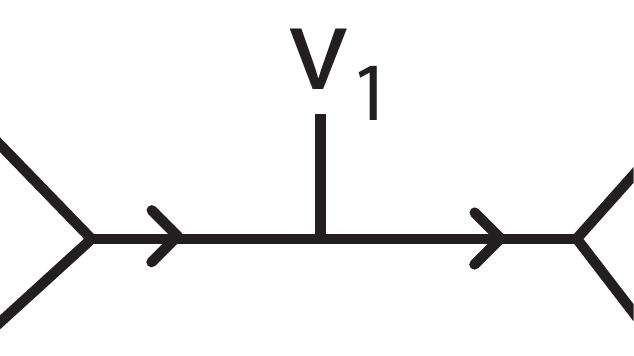}\fi
\end{minipage}\Big)=\begin{minipage}{1.0in}
\ifpdf\includegraphics[width=1.0in]{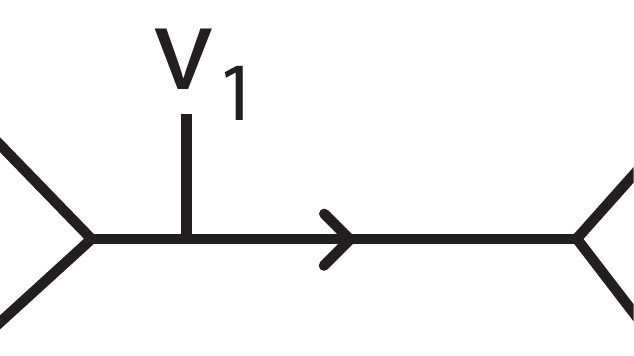}\fi
\end{minipage}
-
\begin{minipage}{1.0in}
\ifpdf\includegraphics[width=1.0in]{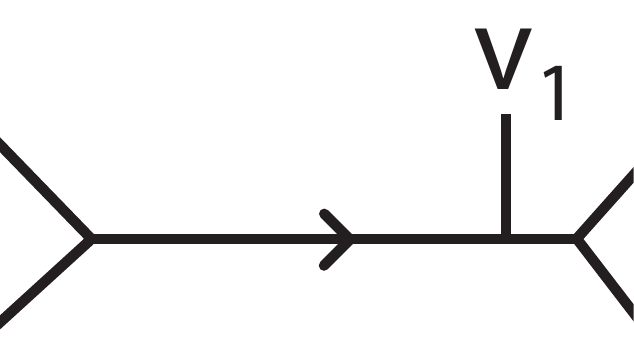}\fi
\end{minipage}
$$

\end{proof}

Recall that the \emph{rank} of a hairy graph is its first Betti number.  Since the boundary operator $\bdry_\hairy$ preserves rank, the  chains $C_k\mathcal \hairy$  decompose  into subcomplexes
$$C_k\mathcal \hairy=\bigoplus_r C_{k,r}  \hairy,$$
where $C_{k,r}\hairy$ is spanned by connected hairy graphs of rank $r$.  On the level of homology this gives
 $$\displaystyle H_k(\hairy) =  \bigoplus_{r\geq 0} H_{k,r}(\hairy),$$
where  $H_{k,r}(\hairy)= H_{k}(C_{*,r}\hairy)$.  The next three propositions give elementary calculations of $H_{1,r}$ for $r\leq 2$.  In the following sections we identify   $H_{1,r}$ for all $r\geq 2$ with a certain twisted cohomology of $\Out(F_r)$ and then calculate this twisted homology for $r=2$ in terms of modular forms.

\begin{proposition}\label{thm:rankzero}
For $\cO=\Lie$ the rank zero part of $H_1(\hairy)$ is
$H_{1,0}(\hairy)\cong \Wedge^3 V=\SF{(1,1,1)}V$.
\end{proposition}
\begin{proof}
A rank $0$   Lie graph has no oriented edges, so  is a union of trees; since we are only looking at $C_1\hairy$ there is only one tree.  If this tree has more than $3$ leaves then it is in the image of $\bdry_\hairy$.  A tripod, on the other hand, cannot be  in the image of $\bdry_\hairy$, so $H_{1,0}(\hairy)$ is spanned by tripods.   If we choose an ordering for the labels of each tripod, the map sending the labels to their wedge product is an isomorphism $H_{1,0}\to \Wedge^3 V$.
\end{proof}

\begin{figure}
\begin{center}
\ifpdf\includegraphics[width=5in]{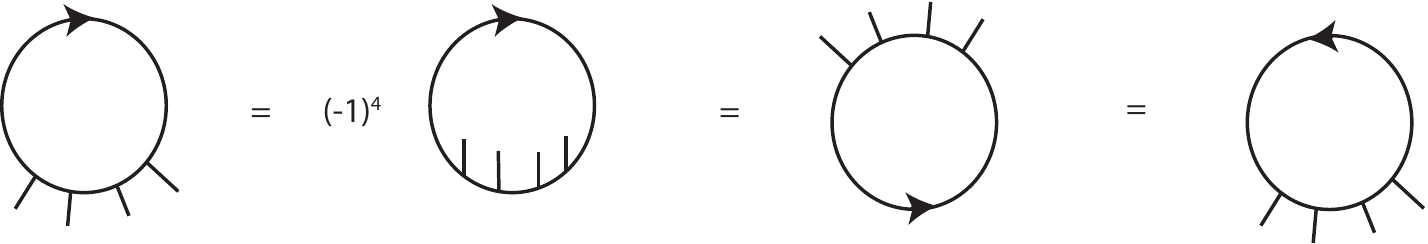}\fi
\end{center}
\caption{Orientation-reversing automorphism}\label{orient}
\end{figure}

\begin{proposition}\label{thm:rankone}
For $\cO=\mathcal Lie$ the rank one part of $H_1(\hairy)$ is
$$
H_{1,1}(\hairy)\cong \bigoplus_{k\geq 0}S^{2k+1}V=\bigoplus_{k\geq 0}\SF{(2k+1)}V.
$$
\end{proposition}
\begin{proof}
Define a map $$\phi\colon C_{1,1}\hairy \to \bigoplus_{k\geq 0}S^{2k+1}V$$ by setting $\phi(\G)=0$ unless $\G$ is a single oriented loop with hairs attached.  For such $\G$, define $\phi(\G)$ to be the product of the labels on its hairs. Note that  the number of hairs must be odd, since otherwise $\G$ has an  orientation reversing automorphism (see Figure~\ref{orient}), giving  $\G=-\G,$ i.e. $\G=0$ in $\hairy$.

The map $\phi$ is clearly  surjective, and we claim it induces an isomorphism on homology, i.e. that $\ker(\phi)=\im(\bdry_\hairy)$. Any graph which is not a hairy loop is in $\im(\bdry_\hairy)$ by Lemma~\ref{separatingLie} and $\phi$ is injective on hairy loops, so $\ker(\phi)\subseteq \im(\bdry_\hairy)$.  To see the opposite inclusion, let ${\bf H}$ be a generator of $C_{2,1}\hairy$. Then $\bf H$ is represented by two trivalent planar trees joined by one oriented edge which connects them, plus another oriented edge.   If the second oriented edge has both ends in one tree, then $\bdry_\hairy(\bf H)$ consists of one graph with a separating unoriented edge, and $\phi(\bdry_\hairy(\bf H))=0$.  If the second oriented edge also joins the two trees, then $\bdry_\hairy(\bf H)$ has two terms, each an oriented loop with trees attached. By Lemma~\ref{order}, they are actually the same graph modulo boundaries. However, the signs are opposite, so $\bdry_\hairy(\bf H)=0$.
\end{proof}
We remark that it is not difficult to show that $\Tr_*$ maps onto $\displaystyle\bigoplus_{k\geq 1}S^{2k+1} V\subset H_{1,1}(\hairy)$, which recovers the part of abelianization constructed by Morita in \cite{Morita}.

\begin{proposition}\label{thm:ranktwo}For $\cO=\Lie,$ $H_{1,2}(\hairy)$ is isomorphic to the subspace of the polynomial ring $\F[V\oplus V]$ characterized by the conditions:
\begin{enumerate}
\item $f(x,y)=f(y,x)$;
\item $f(x,y)=-f(-x,y)$;
\item $f(x,y)+f(y,-x-y)+f(-x-y,x)=0$.
\end{enumerate}
\end{proposition}
\begin{proof}
The chain group $C_{1,2}(\hairy)$ is generated by rank $2$ trivalent graphs $\G$ with hairs attached.  A maximal tree (in this case a single edge) is specified, and the other edges are oriented.  To calculate $H_{1,2}(\hairy)$ we need to account for the relations in $C_{1,2}(\hairy)$ arising from IHX and antisymmetry and calculate the image of $\bdry_\hairy(C_{2,2}(\hairy))$.

There are only two trivalent graphs in rank 2:  the theta graph and the eyeglass graph, so $C_{1,2}$ decomposes as $Tri \oplus Gla$, where $Tri$ is generated by theta graphs and $Gla$ is generated by eyelass graphs.
Any hairy graph based on the eyeglass graph is a boundary by Lemma~\ref{separatingLie}, i.e. $Gla\subset \im(\bdry_\hairy)$.

Using IHX, we can push the hairs off of the tree edge of $\G\in Tri$, decomposing $\G$ as a sum of theta graphs with hairs only on the oriented edges.
If there is an even number of hairs on one of these edges, the IHX relation using the tree edge, together with antisymmetry relations, shows that $\G$ is zero modulo boundaries:  one term of the IHX relation is based on an eyeglass graph, and the other is equal to the first, giving $2\G=0$ (see Figure~\ref{IHX2G}).

\begin{figure}
\begin{center}
\ifpdf\includegraphics[width=4in]{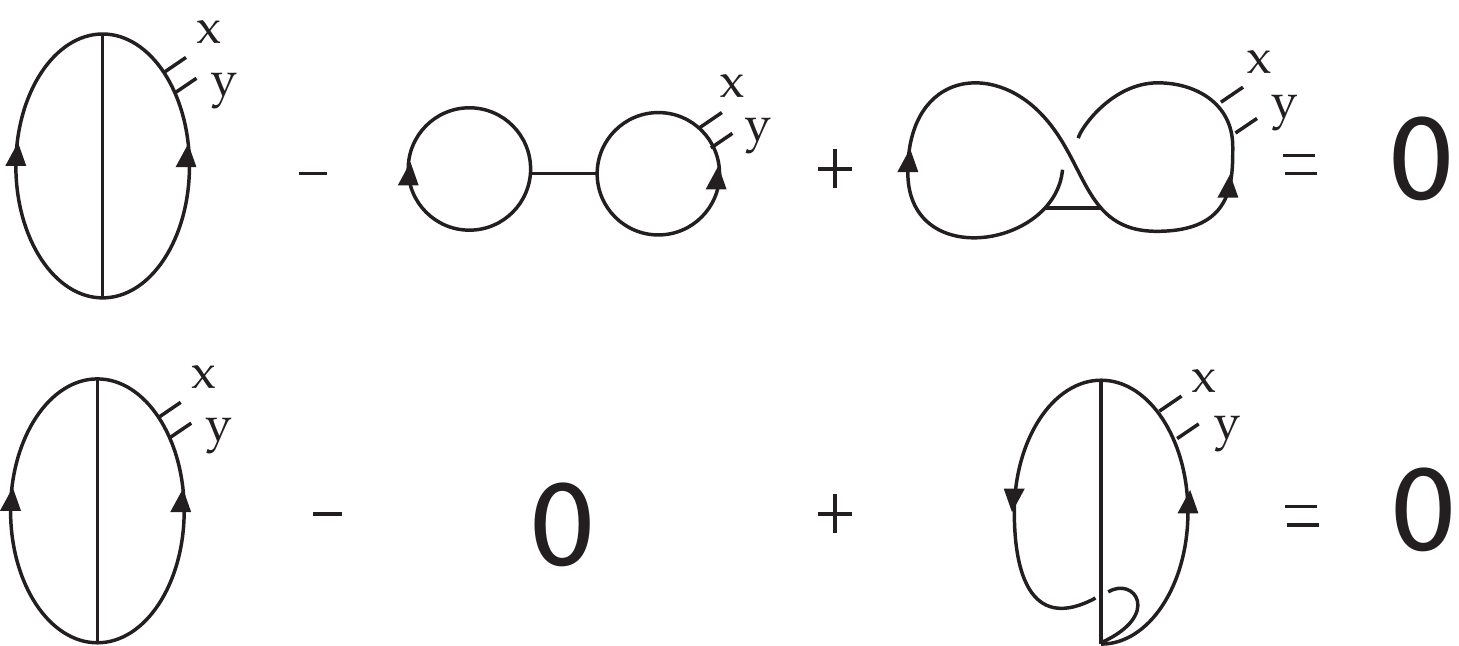}\fi
\caption{IHX relation on hairy theta graph}\label{IHX2G}
\end{center}
\end{figure}

Using anti-symmetry relations, we can make $\G$ planar, put the tree edge in the center and flip each hair to the right-hand side of its oriented edge.  We then associate  to each edge   the monomial formed by multiplying the labels on the hairs. We consider each of these as a monomial in a separate copy of $V$, one with variables $x$ and one with variables $y$, and form their product $f(x,y)$.
By Lemma~\ref{order} the order the hairs are attached is irrelevant modulo boundaries, and the monomial $f(x,y)$ completely determines $\G$. The fact that the number of hairs on each oriented edge is odd  means  the degree of $f(x,y)$ is odd in $x$ and in $y$, which can be rewritten as condition (2).

The symmetry of the theta graph together with the anti-symmetry relation in $\Lie$ imposes  condition  (1).

Finally, condition (3)  identifies  the rest of the image of $\bdry_\hairy$.   Let $\G'$ be a generator of $C_{2,2}$ based on the theta graph.  Then $\G'$   consists of two tripods connected by three oriented edges, and $\bdry_\hairy(\G')$   is a sum of three terms.   Pushing the hairs off of the tree edge in each term corresponds exactly to forming the summands of the third condition.
\end{proof}

\begin{example}\label{ex:ftwok}
Let $x_1,\ldots,x_n,y_1,\ldots,y_n$ be coordinate functions for $V\oplus V$, where $x_i$ represent the first factor and and $y_i$ the second. Suppose $k\geq 2$. Then define a polynomial function
$$f_{2k}(x_1,\ldots,x_n,y_1,\ldots,y_n)=x_1y_2^{2k-1}-x_2y_1y_2^{2k-2}+y_1x_2^{2k-1}-y_2x_1x_2^{2k-2}.$$
One may verify that $f_{2k}$ satisfies the three conditions above, so it represents a nontrivial homology class for $\hairy_{1,2}$ with $2k$ hairs. In particular, this picks up the first degree in which $\hairy_{1,2}\neq 0$, when the number of hairs is $4$. In section~\ref{subsec:autclasses},  we will see that $f_{2k}$ is connected to the Eisenstein series, at least for low values of $k$.
\end{example}

\subsection{$H_{1,r}(\hairy)$ and twisted cohomology of $\Out(F_n)$}
In this section and the next we will give deeper insight into the results of the calculation of $H_{1,2}$, as well as giving a general formula for $H_{1,r}$.  This general formula is in terms of the cohomology of $\Out(F_r)$ with coefficients in the polynomial ring $\F[V^{\oplus r}]$ = $\F[V\otimes \F^{r}]$,
where the action is via the quotient $\Out(F_r) \to \GL_r$ and the standard action
of $\GL_r$ on $\F^{r}$.
We begin by explaining how this is computed, in order to relate it to hairy graph homology.  For a detailed explanation of the relation between (unhairy) Lie graph homology and the cohomology of $\Out(F_r)$ with trivial coefficients we refer to~\cite{CV}, section 3.

The group $\Out(F_n)$ acts on a contractible cube complex $K_n$, called the \emph{spine of Outer space} (see~\cite{CV86}).  Stabilizers of this action are finite, so by a standard argument (see, e.g. Brown's book~\cite{brown}), the quotient $K_n/\Out(F_n)$ has the same cohomology as $\Out(F_n)$ with trivial rational coefficients.    The argument   adapts easily to the case of non-trivial coefficients in a rational representation as follows:

\begin{proposition}\label{brown}
Let $X$ be a contractible $CW$ complex on which a group $G$ acts with finite point stabilizers, let $C^*(X)$ be the cellular cochain complex for $X$, and suppose that $M$ is a $G$-module which is a vector space over a field of characteristic $0$. Then $C^*(X)\otimes_G M$ is a cochain complex computing $H^*(G;M)$.
\end{proposition}
\begin{proof}
We follow the discussion from~\cite{brown}. Namely, on p.174, equation 7.10, there is a first-quadrant spectral sequence
$$E^1_{pq}=\bigoplus_{\sigma\in\Sigma_p} H_q(G_\sigma,M_\sigma)\Rightarrow H_*(G,M)$$
where $\Sigma_p$ is a set of representatives of $G$-orbits of $p$-cells, $G_\sigma$ is the stabilizer of $\sigma$ and $M_\sigma$ is $M$ twisted by the ``orientation character." Dually, there is a first-quadrant spectral sequence $$E_1^{pq}=\bigoplus_{\sigma\in\Sigma_p} H^q(G_\sigma,M_\sigma)$$ converging to $H^*(G,M)$.
But $G_\sigma$ is finite, and $M_\sigma$ is a $\Q G_\sigma$-module, so $H^q(G_\sigma,M_\sigma)=0$ for $q>0$ (see, e.g. Corollary 10.2, p. 84 of \cite{brown}). Thus $E_1^{p,q}=0$ for all $q>0$, i.e. the spectral sequence collapses  to simply a cochain complex in the row $q=0$.

Now observe that $E_1^{p,0}=\bigoplus_{\sigma\in \Sigma_\sigma}H^0(G_\sigma, M_\sigma) =\bigoplus_{\sigma\in \Sigma_\sigma}(M_\sigma)_{G_\sigma}= C^p(X)\otimes_G M$.
\end{proof}

For any vector space $W$, denote by $\F[W]$ the ring of polynomial functions on $W$. Note that $\F[W]$ is graded by polynomial degree, i.e. $\F[W]=\bigoplus_k \F[W]_k$, where $\F[W]_k=S^kW$ denotes homogeneous polynomials of degree $k$.

\begin{theorem}\label{thm:hairylie}
For $\cO=\Lie$  and $r\geq 2$ there is a graded isomorphism
$$
H_{1,r}(\hairy)\cong H^{2r-3}(\Out(F_r);\F[V^{\oplus r}]),
$$
where $H_{1,r}(\hairy)$ is graded by the number of hairs, and the grading on $H^{2r-3}(\Out(F_r);\F[V^{\oplus r}])$ is given by polynomial degree:  $$H^{2r-3}(\Out(F_r);\F[V^{\oplus r}])\cong \bigoplus_{k\geq 0} H^{2r-3}(\Out(F_r);\F[V^{\oplus r}]_k)$$
\end{theorem}

\begin{proof}
By Proposition~\ref{brown} applied to the spine $K_r$ of Outer space,  $H^{2r-3}(\Out(F_r);M)$ can be computed using the cochain complex  $C^*=C^*K_r\otimes_{\Out(F_r)} M.$
Recall from~\cite{HV} that each $k$-dimensional cube of $K_r$ is determined by a graph  $G$ equipped with a $k$-edge subforest  $\Phi$ and a \emph{marking},  which is a homotopy equivalence $g$ from $G$ to a fixed rose $R_n$  whose  petals are identified with the generators of $F_n$. The cube $(G, \Phi, g)$ is oriented by ordering the edges of the forest $\Phi$.
The coboundary operator is a sum of two operators $\delta_E$ and $\delta_C$, which add an edge to the forest in all possible ways and expand a vertex into a forest edge in all possible ways, respectively.

The top-dimensional cubes of $K_r$ correspond to marked trivalent graphs with maximal trees, so are $(2r-3)$-dimensional.  The $(2r-4)$-dimensional cubes correspond either to trivalent graphs or to graphs with one $4$-valent vertex.  Using $\delta_C$ to expand the 4-valent vertex in the three possible ways gives the terms of the IHX relation, so the quotient
$$
\bar C^{2n-3}K_r= C^{2n-3}K_r/\im(\delta_C)
$$
is generated by trivalent marked graphs modulo IHX relations using edges of their maximal trees, and
$$
H^{2r-3}(\Out(F_r);\F[V^{\oplus r}]) =
\bar C^{2n-3}K_r\otimes_{\Out(F_r)} \F[V^{\oplus r}]/\im(\delta_E\otimes 1).
$$

We now turn to the hairy graph homology computation
$$
H_{1,r}(\hairy) = C_{1,r}\hairy/\im(\bdry_\hairy).
$$
A generator $\G$ of $C_1\hairy$ can be represented (modulo anti-symmetry and IHX relations) by a planar trivalent tree with some pairs of leaves joined by oriented edges and the rest labeled by elements of $V$.  If $\G$ has a separating unoriented edge which is not a hair then $\G$ is a boundary by Lemma~\ref{separatingLie}.   If all separating edges are hairs,  then removing these hairs results in a trivalent \emph{core graph} $G$.  If any hairs are on unoriented edges of $G$, then they may be moved using IHX relations to the oriented edges.  Using anti-symmetry relations, we may flip each hair to the right-hand side of its oriented edge.
Thus, as generators for $H_{1,r}(\hairy)$ we may take trivalent graphs $G$ of rank $r$ such that the unoriented edges form a maximal tree $T$ and the oriented edges $\vec e$ have  labeled hairs attached to the right-hand side.  Modulo $\hbox{im}(\bdry_\hairy)$, the order of the hairs on each oriented edge does not matter.

We  can now define the isomorphism
$$
f\colon C_{1,r}\hairy/\hbox{im}(\bdry_\hairy) \to \bar C^*K_r\otimes_{\Out(F_r)} \F[V^{\oplus r}]/\hbox{im}(\delta_E\otimes 1).
$$
Let $\G$ be a generator of $C_{1,r}\hairy$, as described above.   To get a marking $g\colon G\to R_n$, we collapse the unoriented edges of $G$ to obtain a rose $G/T$, then choose a homeomorphism from $G/T$ to the standard rose $R_n$ preserving the orientations on the edges.    If $g(\vec e)=x_i$  set $m_i\in \F[V]$ equal to the product of the labels of the hairs on $\vec e$.  Then
$$
f(\G)=(G, T, g)\otimes m_1\ldots m_r\in \F[V^r].
$$

This map is well-defined and surjective; in particular it does not depend on the choice of the homeomorphism from $G/T$ to $R_n$ since the symmetric group permuting the petals of $R_n$ is a subgroup of $\Out(F_r)$.  To see that it is injective, note that $\bdry_\hairy$ coincides with $\delta_E$ under this map.
\end{proof}

\begin{remark}
This proof does not work to compute $H_i(\hairy)$ for $i>1$ since we allowed ourselves to slide hairs  across oriented edges using  Lemma~\ref{slideLie}. Unfortunately, there is no analogue of Lemma~\ref{slideLie} for hairy graphs in $C_i\hairy$ with $i>1$.
\end{remark}

\subsection{$H_{1,2}(\hairy)$ and modular forms}
In this section, we let $\F=\C$.
For $r=2$  we have identified
$$
H_{1,2}(\hairy) \cong H^1(\Out(F_2), \C[V\oplus V]).
$$

Since the abelianization map $F_2\to \Z^2$ induces an isomorphism $\Out(F_2)\cong \GL_2(\Z)$  we can use the representation theory of $\GL_2(\Z)$ to calculate this group precisely.  The answer involves the dimension $s_k$ of the space of weight $k$ cuspidal modular forms for $\SL_2(\Z),$ which is zero if $k$ is odd or if $k=2$. For $k>2$ even, it is given by
 $$s_k=\begin{cases}
\lfloor k/12\rfloor-1 &\text{ if } k\equiv 2\mod 12\\
\lfloor k/12\rfloor &\text{ if } k\not\equiv 2\mod 12
\end{cases}.
$$
Recall also that the Weyl module $\SF{(k,\ell)}V$ is the irreducible representation of $\GL(V)$ corresponding to the partition $(k,\ell)$.

\begin{theorem}
\label{thm:modular}
There is a graded isomorphism
$$
H_{1,2}(\hairy)\cong \bigoplus_{k>\ell\geq 0} (\SF{(k,\ell)}V)^{\oplus \lambda_{k,\ell}}
$$
where $\lambda_{k,\ell}=0$ unless $k+\ell$ is even, in which case $\lambda_{k,\ell}=\begin{cases}s_{k-\ell+2} &\text{if }\ell \text{ is even}\\
s_{k-\ell+2}+1&\text{if }\ell \text{ is odd}.
\end{cases}$

The grading on $H_{1,2}(\hairy)$ is by the number of hairs (=degree - 2) and on $ \bigoplus_{k>\ell\geq 0}(\SF{(k,\ell)}V)^{\oplus \lambda_{k,\ell}}$ is by $k+\ell$.
\end{theorem}

\begin{proof}
By Theorem~\ref{thm:hairylie}, $H_1(\hairy)\cong
H^1(\Out(F_2);\C[V\oplus V])$.
Since the natural map from $\Out(F_2)$ to $\GL_2(\Z)$ is an isomorphism, we may instead compute
$H^1(\GL_2(\Z);\C[V\oplus V]).$

Set $P_V=\C[V\oplus V]=\C[V\otimes \C^2]$.
Then $P_V$ is a $\GL(V)\otimes \GL_2(\C)$-module, which by Schur-Weyl  duality can be decomposed as
$$
P_V=\bigoplus_\lambda\, \SF{\lambda} V\otimes \SF{\lambda}\C^2.
$$
(See~\cite{GW} p.218 and p.257.)  The Weyl module $\SF{\lambda}\C^2$ is zero unless $\lambda=(k,l)$ is a Young diagram with only two rows, and
$$
\SF{(k,l)}\C^2 = \C_{l}\otimes H_{k-l}
$$
where $H_{k-l}$ is the space of homogeneous polynomials of degree $k-l$, and $\C_{l}$ is the one-dimensional $\GL_2(\C)$-representation given by the $l$th power of the determinant (see~\cite{FH}, Section 6.1).

Thus
\begin{align}
H^1(\GL_2(\Z); P_V)&=
H^1\left(\GL_2(\Z); \bigoplus_{k\geq l\geq 0}
\SF{(k,l)}V\otimes \C_{l} \otimes H_{k-l}\right)\\
&\cong\bigoplus_{k\geq l\geq 0}
H^1(\GL_2(\Z); \C_{l} \otimes H_{k-l})\otimes \SF{(k,l)}V.\notag
\end{align}

The cohomology 5-term exact sequence of the extension
$$
1\to \SL_2(\Z) \to \GL_2(\Z)\to\mathbb \Z_2\to 1
$$
reads
$$
0 \to H^1(\Z_2,\C_{l} \otimes H_m)\to H^1(\GL_2(\Z),\C_{l} \otimes H_m)\to H^1(\SL_2(\Z),\C_{l} \otimes H_m)^{\Z_2}\to \hbox{\hskip 1.5in}
$$
$$
\hbox{\hskip 2.5in}H^2(\Z_2,\C_{l} \otimes H_m)\to H^2(\GL_2(\Z),\C_{l} \otimes H_m).
$$
Since $H^1(\Z_2;M)=H^2(\Z_2;M)=0$ for any vector space $M$
over a field of characteristic $0$
this gives
$$
H^1(\GL_2(\Z); \C_{l} \otimes H_m)\cong
\left(H^1(\SL_2(\Z); \C_{l} \otimes H_m)\right)^{\Z_2}
\cong
\left( \C_{l} \otimes H^1(\SL_2(\Z);H_m)\right)^{\Z_2},
$$
where the generator of $\mathbb Z_2$ acts on $\C_{l}$ via multiplication by $(-1)^l$.

The computation is now completed using Eichler-Shimura theory
(see, e.g.~\cite[p. 246-247]{Hab}, ). The action of $\Z_2$  is induced by conjugation by $\epsilon=\begin{pmatrix} -1 & 0 \\ 0 & 1\end{pmatrix}$, and
$$
H^1(\SL_2(\Z);H_m)\cong H^1_+\oplus H^1_-
$$
where $H^1_+$ is the $(+1)$-eigenspace of the action, and $H^1_-$ is the $(-1)$-eigenspace.  These eigenspaces are given by
$$
H^1_+ \cong M^0_{m+2},
$$
where $M^0_{m+2}$ is the vector space of weight $m+2$ cuspidal modular forms  for the full modular group $\SL_2(\Z)$ and
$$
H^1_-\cong  {M}^0_{m+2}\oplus E_{m+2},
$$
where  $E_{m+2}=0$ if $m$ is odd or if $m=0$, and otherwise  $E_{m+2}$ is the one dimensional space spanned by the Eisenstein series in degree $m+2$.    Therefore, if the action of $\GL_2(\Z)$ is standard, we get
 $$H^1(\GL_2(\Z); H_m)\cong H^1(\SL_2(\Z);H_m)^{\Z_2}\cong H^1_+\cong M^0_{m+2}$$
 and if the action is twisted by the determinant we get
 $$H^1(\GL_2(\Z); H_m)\cong H^1(\SL_2(\Z);H_m)^{\Z_2}\cong H^1_-\cong M^0_{m+2}\oplus E_{m+2}.$$
Since $M\otimes \SF{(k,l)}V\cong (\SF{(k,l)}V)^{\dim M}$, plugging this result into expression (1) above   gives the theorem. Notice that the case $m=0$ is special since $E_2=0$. In this case $ H^1(\GL_2(\Z); H_0)=0$, so that partitions where $k=\ell$ do not contribute.
\end{proof}

Here is a table of all Weyl modules which appear in $H_{1,2}(\hairy)$  for graphs with at most $14$ hairs, i.e. $k+l\leq 14$.   The notation $(k,\ell)^m$ means that $\SF{(k,\ell)}V$ appears $m$ times.
\begin{center}
\begin{tabular}{c|c|c|c|c|c|c}
$2$&$4$&$6$&$8$&$10$&$12$&$14$\\
\hline
&$(3,1)$&$(5,1)$&$(7,1)$&$(10,0)$&$(11,1)^2\rule{0pt}{3.5ex}$&$(14,0)$\\
&&&$(5,3)$&$(9,1)$&$(9,3)$&$(13,1)$\\
&&&&$(7,3)$&$(7,5)$&$(12,2)$\\
&&&&&&$(11,3)$\\
&&&&&&$(9,5)$
\end{tabular}
\end{center}

\section{Comparison to Morita's trace}
In this section, we describe Morita's trace map \cite{Morita} using graphical insights from \cite{CVMorita}, and see how it relates to the trace map $\Tr$ defined in this paper for the Lie case.
 It is defined on $\lplus_V$ as follows: sum over connecting pairs of univalent vertices by an edge (direction arbitrarily fixed) and multiply by the contraction of the coefficients of the involved univalent vertices. (See Figure~\ref{fig:moritatrace}.)
 This yields a graph with a loop and some attached trees. Each tree represents an element of the free Lie algebra over $V$, which we can include in the free associative algebra. Thus, reading around the circle in the direction indicated by the edge's direction, we obtain an element of the free associative algebra $T(V)$. The free associative algebra has an involution defined on words by $w\mapsto (-1)^{|w|}\bar w$, where $|w|$ is the length of the word and $\bar w$ is the reversal of the word. In order to account for the ambiguity of the circle's orientation, the image of this map will take values in $T(V)$ modulo this involution. Now project to the free commutative algebra generated by $V$: $\bigoplus_{k=0}^\infty S^kV$. Dividing this by the image of the involution on $T(V)$, we are left with only the odd powers
$ \bigoplus_{k=0}^\infty S^{2k+1}(V)$, and this is the target of Morita's trace map:
$$\Tr^M\colon \lplus_V\to  \bigoplus_{k=0}^\infty S^{2k+1}(V) $$

\begin{figure}[h]
$$
\begin{minipage}{2.7cm}
\begin{tikzpicture} [thick]
\draw (0,0) -- (-.5,.866) node [fill=white] {$v_1$};
\draw (0,0) -- (-.5,-.866) node [fill=white] {$v_2$};
\draw (0,0) -- (.5,0);
\draw (.5,0)--(.5,.866);
\draw (.5,.866)--(0,1.732) node [fill=white]{$v_3$};
\draw (.5,.866)--(1.0,1.732) node [fill=white]{$v_4$};
\draw (.5,0)--(1.0,0);
\draw (1.0,0)--(1.5,.866) node [fill=white] {$v_5$};
\draw (1.0,0)--(1.5,-.866) node [fill=white] {$v_6$};
\end{tikzpicture}
\end{minipage}\mapsto
\omega(v_2,v_6)
\begin{minipage}{2.7cm}
\begin{tikzpicture} [thick]
\draw (0,0) -- (-.5,.866) node [fill=white] {$v_1$};
\draw (0,0) -- (-.5,-.866);
\draw (0,0) -- (.5,0);
\draw (.5,0)--(.5,.866);
\draw (.5,.866)--(0,1.732) node [fill=white]{$v_3$};
\draw (.5,.866)--(1.0,1.732) node [fill=white]{$v_4$};
\draw (.5,0)--(1.0,0);
\draw (1.0,0)--(1.5,.866) node [fill=white] {$v_5$};
\draw (1.0,0)--(1.5,-.866);
\draw[densely dotted, ->] (1.5,-.866) arc (0:-90:1); 
\draw[densely dotted] (-.5,-.866) arc (180:270:1);
\end{tikzpicture}
\end{minipage}
+\cdots
\mapsto \omega(v_2,v_6) v_1[v_3,v_4]v_5+\cdots
$$
\caption{The first step in defining Morita's trace map. Sum over adding an edge to the tree in all possible ways, and read off the resulting element in the free associative algebra.}\label{fig:moritatrace}
\end{figure}
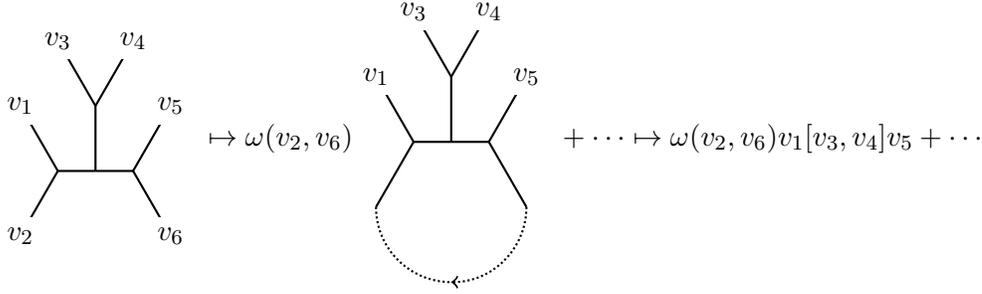

 After stabilizing,
 it is not difficult to show that $\Tr^M$ is surjective, and that any bracket of two trees is in the kernel of this trace, so that this actually gives rise to a very interesting abelian quotient of the Lie algebra $\lplus_\infty$. The $k=0$ summand corresponds to trees in $\lplus_\infty$ with a single trivalent vertex, which cannot be brackets of smaller trees by degree reasons. Such trees form an isomorphic copy of $\ext^3 V_\infty$ inside of $(\lplus_\infty)_{\rm ab}$, so $\Tr^M$ is not an isomorphism at this bottom degree (degree $1$). However, if we replace $S^1V_\infty=V_\infty$ by $\ext^3 V_\infty$ in the above direct sum, Morita conjectured that this is isomorphic to the entire abelianization~\cite[Conjecture 6.1]{Morita}. 
  
Consider the middle term in Figure~\ref{fig:moritatrace}. This was a convenient graphical midway point in calculating the Morita trace, but we now adopt the point of view that the vector space spanned by trees with an additional (dashed) edge is actually the natural target of the Morita trace. Indeed such graphs form a subspace  of the hairy graph complex $\mathcal H$.  With this point of view, it is not too hard to show that Morita's trace map induces a map from the abelianization to $H_1(\mathcal H)$, which lands in the subspace of rank $1$ graphs.
 
The preceding discussion can be summarized by the following proposition. 
\begin{proposition}
The Morita trace $\Tr^M$ is equal to the composition 
$$
\ext^1 \lplus_\infty\hookrightarrow \ext \lplus_\infty\overset{\Tr}{\longrightarrow}\hairy\twoheadrightarrow \hairy_{1,1}.
$$
\end{proposition}

\section{Cycles in the unstable homology of $\Mod(g,s)$, $\Out(F_n)$ and $\Aut(F_n)$}

For each cyclic operad $\cO$ and symplectic vector space $V$, the abelianization map $\lplus_V\to \lplus_V^{\text{ab}}$ is a Lie algebra morphism, where the bracket on $\lplus_V^{\text{ab}}$ is trivial.  If  $\cO$ is  finite dimensional at each level continuous cohomology is defined (see Definition~\ref{cont}) so abelianization induces a backwards map   $H^*_c(\lplus_V^{\text{ab}})\to H^*_c(\lplus_V)$. Taking $\SP$-invariants gives a map
\begin{equation}
\label{eqn:abelianization}
PH^*_c(\lplus_V^{\text{ab}})^{\SP}\to PH_c^*(\lplus_V)^{\SP},
\end{equation}
where $P$ denotes the submodule of primitive elements in the Hopf algebra $H_c^*(\lplus_V)^{\SP}$. Using the fact that the cohomology of a finite-dimensional abelian Lie algebra is simply the exterior algebra on its dual space, the domain of this map is often easy to compute and gives rise to potentially non-trivial elements in the image.

For  $\cO = \Assoc$ and $\cO=\Lie$  a theorem of Kontsevich identifies $PH_c^*(\lplus_V)^{\SP}$ for infinite-dimensional $V$ with the homology of mapping class groups of punctured surfaces and outer automorphism groups of free groups, and a theorem of Gray relates similar cohomology groups to the homology of $\Aut(F_n)$. In this section we show how to exploit these theorems together with the map (\ref{eqn:abelianization}) above to construct cycles for the homology of these groups.

\subsection{Mapping class groups of punctured surfaces}
Let $\Mod(g,s)$ denote the mapping class group of a surface of genus $g$ with $s$ punctures, i.e. the group of isotopy classes of homeomorphisms which preserve the set of punctures (not necessarily pointwise).  Set $V_n=\F^{2n}$ with the standard symplectic form and $V_\infty=\dirlim \F^{2n}$;  write $\lplus_n$ for $\lplus_{V_n}$ and $\lplus_\infty=\lplus_{V_\infty}$.  Kontsevich's theorem for $\cO=\Assoc$ reads:
\begin{theorem}\cite{CV,Ko1,Ko2}
For $\cO=\Assoc$,
$$ PH_c^k(\lplus_{\infty})^{\SP}\cong\bigoplus_{s>0}H_{4g+2s-k-4}(\Mod(g,s);\F).
$$
\end{theorem}

To use the map~(\ref{eqn:abelianization}) to find classes in $H_*(\Mod(g,s))$ we must now compute $PH^*_c(\lplus_\infty^{\text{ab}})^{\SP}$.
By Theorem~\ref{thm:associative} we have that $\lplus_V^{\rm ab}\cong W_1\oplus W_2$ where
$W_1=[V^{\otimes 3}]_{\Z_3}$ and  $W_2=(\ext^2 V)/\F$.
In~\cite{Morita2}  Morita calculated that if $\dim V \gg k$,
$$
P\left(\ext^k W_2\right)^{\SP}=\left(\ext^k W_2\right)^{\SP}\cong \begin{cases}
\F &\text{if } k\equiv 1\mod 4 \text{ and } k\geq 5\\
0&\text{otherwise}
\end{cases}.
$$
Since the result of the calculation is independent of $V$ we can take duals on the finite level and conclude
that $PH^{4r+1}_c(\lplus_\infty^{\text{ab}})^{\SP}$ contains a copy of $\F$ for each  $r\geq 1$.
Applying the map~(\ref{eqn:abelianization}) now gives a cocycle in $PH_c^{4r+1}(\lplus_\infty)^{\SP}$ for each $r\geq 1$, which corresponds via Kontsevich's theorem to a cycle in $H_{4r+1}(\operatorname{Mod}(1,4r+1))$.  In~\cite{Jim} it was shown that all of these cycles in fact represent non-trivial homology classes.

We have only used the degree $2$ piece $W_2$ of the abelianization to construct these homology classes.  Using $W_1$ as well  we can construct  many more cycles; for example it is easy to compute that $\left[(\ext^2 W_1)\otimes (\ext^2 W_2)\right]^{\SP}\neq 0$, giving   $2$-dimensional cycles for $\Mod(1,3)$ and $\Mod(2,1)$.  However,   we do not know whether these cycles are non-trivial in   homology.  In the sequel to this paper we will show how to produce  cycles on moduli space (of any genus) by using classes in $H_k(\hairy)$ for $k>1$, potentially yielding even more unstable homology classes.

\subsection{The outer automorphism group $\Out(F_n)$}
Again we set $V_n=\F^{2n}$ with the standard symplectic form, $V_\infty=\dirlim \F^{2n}$, and  write $\lplus_n$ for $\lplus_{V_n}$ and $\lplus_\infty=\lplus_{V_\infty}$.  Kontsevich's theorem for $\cO=\Lie$ reads:

\begin{theorem}\cite{CV,Ko1,Ko2}\ For $\cO=\Lie$,
$PH_c^k(\lplus_\infty)^{\SP}$ is non-zero only in even degrees $2d$, in which case
$$
PH_c^k(\lplus_\infty)^{\SP}\degree{2d}\cong  H_{2d-k}(\Out(F_{d+1});\F).
$$
\end{theorem}

Following the abelianization map with the trace map yields Lie algebra morphisms
$$\lplus_n\to \lplus_n^{\text{ab}}\hookrightarrow H_1(\hairy_n)$$
where both $\lplus_n^{\text{ab}}$ and $H_1(\hairy_n)$ are thought of as abelian Lie algebras, graded by degree.
In degree $d$, these maps induce backwards maps
$$
\sum_{d_1+\ldots+d_k=d} H_1(\hairy_n\degree{d_1})^*\wedge\ldots\wedge H_1(\hairy_n\degree{d_k})^* \to H^*(\lplus_n)\degree{d},
$$
using  the fact that $H^*(\mathfrak a)=\ext \mathfrak a^*$ for finite-dimensional abelian Lie algebras $\mathfrak a$.
Taking the primitive part of the $\SP$-invariants and letting $n$ go to infinity yields a map
$$
\mu\colon \invlim P(\sum_{d_1+\ldots+d_k=d} H_1(\hairy_n\degree{d_1})^*\wedge\ldots\wedge H_1(\hairy_n\degree{d_k})^*)^{\SP} \to PH_c^*(\lplus_\infty)^{\SP}\degree{d}.
$$
Thus by combining elements of the first homology of the hairy graph complex, we obtain cocycles in $PH_c^*(\lplus_\infty)^{\SP}\degree{d}$, which by Kontsevich's theorem can be identified with cycles in $H_{2d-k}(\Out(F_{d+1});\F)$. We illustrate this with two concrete examples below.

\subsubsection{Morita's original cycles}  Morita's original series of cycles was constructed from elements of  $\lplus^{\text{ab}}_V$ in degree $d=2k-1$.  When pushed by the trace into hairy graph homology, these correspond to $H_{1,1}(\hairy_V)\degree{2k-1} \cong S^{2k-1}V=\SF{2k-1}V$. In hairy graph homology,   generators of $H_{1,1}(\hairy_V)\degree{2k-1}$ are represented by   oriented loops with $2k-1$ hairs attached, labeled by elements of $V$.

A straightforward computation shows that $W_{n,2k-1}:= [(\SF{2k-1}V_n)\wedge (\SF{2k-1}V_n)]^{\SP}\cong \F$ for large enough $n$.
The   generator of $W_{n,2k-1}$ corresponds to two hairy loops, with the hairs on one paired with  the hairs on the other; in particular  the hair labels have disappeared and the generator is independent of $V_n$. Since this graph is connected, it represents a primitive class. Since in a Hopf algebra the dual to the submodule of primitives is primitive with respect to the dual Hopf algebra structure, we get
$W^*_{n,2k-1}\subset
P\left[ (S^{2k-1}V_{n})^*\wedge (S^{2k-1}V_{n})^*\right]^{\SP}$.
Let $W^*_{2k-1}:= \invlim W_{n,{2k-1}}^*$.
The image of the generator of $W^*_{2k-1}$ under the map $\mu$ above is in $PH_c^2(\lplus_\infty)^{\SP}$ and under Kontsevich's theorem corresponds to the $k$-th Morita class, in $H_{4k-4}(\Out(F_{2k});\F)$. See~\cite{CVMorita} for more details on these and other classes arising from the rank one part of the abelianization.

\subsubsection{New classes from cusp forms}\label{sec:newcusp}

Recall that $H_1(\hairy_V)\supset H_{1,2}(\hairy_V)\supset (\SF{(k,l)}V)^{\lambda_{k,l}}$.  We will use the piece with   $(k,l)=(2m,0)$  to construct new cohomology classes in $PH_c^2(\lplus_\infty)^{\SP}$.  In this case the exponent $\lambda_{(m,0)}$ is equal to $s_{2m+2}$, the dimension of the space $M^0_{2m+2}$ of cusp forms of weight $2m+2$, and in fact we have

$$
(\SF{(2m,0)}V)^{2m+2}=M^0_{2m+2}\otimes \SF{(2m,0)}V=M^0_{2m+2}\otimes S^{2m} V.
$$

\begin{lemma}\label{lemma:SpCalculation}
$\left[ (M^0_{2m+2}\otimes \SF{2m} V)\wedge (M^0_{2m+2}\otimes \SF{2m} V)\right]^{\SP}$ is isomorphic to $\ext^2(M^0_{2m+2})$.
\end{lemma}

\begin{proof}
Let $U=M^0_{2m+2}\otimes \SpF{2m} V=\F^{s} \otimes \SpF{2m}V$. In order to compute $\left[\ext^2 U\right]^{\SP}$, we first compute $[U\otimes U]^{\SP}$, and then divide by the alternating $\Z_2$-action.
But notice that
$$
\left[ U\otimes U\right]^{\SP}\cong
(\F^s\otimes \F^s)\otimes \left[\SpF{2m} V\otimes \SpF{2m} V\right]^{\SP}\cong
(\F^s\otimes \F^s)\otimes \F,
$$
since by classical invariant theory, $\left[\SpF{2m} V\otimes \SpF{2m} V\right]^{\SP}$ is one-dimensional, generated by $\omega^{2m}$ for $\omega=\sum (p_i\otimes q_i- q_i\otimes p_i)$.

Now to calculate $\left[\ext^2 U\right]^{\SP}$, we take the $\mathbb Z_2$ invariants. $\mathbb Z_2$ acts on the four-fold tensor product $\F^s\otimes \F^s\otimes \SpF{2m} V\otimes \SpF{2m} V$ by the rule $a\otimes b\otimes c\otimes d\mapsto -b\otimes a\otimes d\otimes c$. Swapping the tensor factors of $\omega$ sends it to $-\omega$, so the $\mathbb Z_2$ action on the invariants is $$v\otimes w \otimes \omega^{2m}\mapsto -w\otimes v\otimes (-\omega)^{2m}$$
 Thus, we get
$[(\F^s\otimes \F^s)\otimes \F]^{\Z_2}=\ext^2(\F^s)=\ext^2(M^0_{2m+2})$.
\end{proof}

Generators of   $\ext^2(M^0_{2m+2})$  are represented in hairy graph homology by two rank two hairy graphs with $2m$ hairs each; the hairs on one are paired with the hairs on the other, resulting in a connected graph of rank $2m+3$ (and degree $4m+4$).  So we get
$$\ext^2(M^0_{2m+2})^*\subset P\left[H_1(\hairy_n\degree{2m+2})^*\wedge H_1(\hairy_n\degree{2m+2})^*\right]^{\SP}.$$
Since this is independent of $n$, applying the map $\mu$ together with Kontsevich's theorem yields the following result.
\begin{theorem}\label{thm:symmod}
There is an injection $\ext^2\left(M^0_{2k}\right)^*\hookrightarrow Z_{4k-2}(\Out(F_{2k+1});\F)$ into cycles for $\Out(F_{2k+1})$.
\end{theorem}

The first $M^0_{2k}$ with dimension at least $2$ occurs when $k=12$, yielding a cycle in $Z_{46}(\Out(F_{25}))$. This is well beyond the range in which we can compute whether this is a nonzero homology class.

\subsection{The automorphism group $\Aut(F_n)$}\label{subsec:autclasses}  Kontsevich's theorems for mapping class groups and outer automorphism groups of free groups have been adapted by Gray~\cite{Gray} to yield information about the homology of automorphism groups of free groups.  The basic modification needed in hairy graph homology is to add a distinguished hair which marks a basepoint for the graph.  To keep track of the internal vertex adjacent to the distinguished hair, we think of the operad element coloring the vertex as a coefficient (with a distinguished vertex).  This complicates the algebra somewhat, as we now explain.

Let $L_V$ denote the submodule of the free Lie algebra on $V$ spanned by elements of degree at least $2$. Then $\lplus_V$ acts on $L_V$ by derivations, and we can form the homology groups $H_*(\lplus_V; L_V)$.  The homology $H_*(\lplus_V; L_V)$ is not a Hopf algebra, but it is a \emph{Hopf module} over $H_*(\lplus_V)$, where  in general $M$ is said to be a Hopf module over the Hopf algebra $H$ if there are maps
$H\otimes M\to M$ (this is the module structure) and $M\to H\otimes M$ (this is the comodule structure) satisfying various compatibility axioms. The Hopf module structure for  $H_*(\lplus_V; L_V)$ is defined similarly to the structure for homology with trivial coefficients as detailed in \cite{CV}; we refer to \cite{Gray} for details.

Primitives in a Hopf module are defined to be solutions of the equation $\Delta(x)=1\otimes x$, where $\Delta$ is the coaction. Just as in the case of Hopf algebras, the dual of a Hopf module is also a Hopf module, and primitives get sent to primitives when taking duals. As in the case of trivial coefficients, primitivity translates to connectedess of graphs on the graph homology level.

\begin{theorem}\label{authm}\cite{Gray} For $\cO={\cL}ie,$
 $$PH_k(\lplus_\infty; L_\infty)^{\SP}\cong \bigoplus_{r\geq 2} H^{2r-1-k}(\Aut(F_r);\Q),$$where $\lplus_\infty$ acts on $L_\infty$ by derivations.
\end{theorem}

We want to use the dual form of this isomorphism, so we pause to introduce some notation. Let $M_n=\oplus M_n\degree{d}$ be a graded vector space where each graded summand is finite dimensional. Suppose
$\cdots\to M_{n}\to M_{n+1}\to M_{n+2}\cdots$
is a sequence of graded linear maps. Define $M_\infty^*$ to be the graded dual of $\lim_{n\to\infty} M_n$, i.e.
$$
M_\infty^*:=  \bigoplus_{d}{\invlim}M_n\degree{d}^* .
$$
The dual statement of Theorem~\ref{authm} is then
$$
PH_c^k(\lplus_\infty;L_\infty^*)^{\SP}\cong \bigoplus_{r\geq 2} H_{2r-1-k}(\Aut(F_{r});\Q).
$$

The action of $\lplus_\infty$ on $L_\infty$ induces an action of $\lplus_\infty^{\text{ab}}$ on
$\overline{L}_\infty=L_\infty/(\lplus_\infty\cdot L_\infty)$.
Abelianization $\lplus_\infty\to\lplus_\infty^{ab}$
then induces a backwards map on continuous cohomology  $$
H_c^*(\lplus_\infty^{\text{ab}}; \overline{L}_\infty^*)\to H_c^*(\lplus_\infty;L_\infty^*),
$$
where we emphasize that
$L_\infty^*:=\displaystyle \bigoplus_d {\invlim} L_{V_n}^*\degree{d}$ and similarly
$\overline{L}_\infty^*:=\displaystyle \bigoplus_d {\invlim} \overline{L}_{V_n}^*\degree{d}$.
Taking $\SP$-invariants and then primitives gives
$$PH_c^*(\lplus_\infty^{\text{ab}}; \overline{L}_\infty^*)^{\SP} \to PH_c^*(\lplus_\infty;L_\infty^*)^{\SP}$$

We now want to relate the domain of this map to hairy graph homology.    To this end,   let $V'$ be the vector space generated by $V$ and an additional hyperbolic pair of vectors $b$ and $b^*$, and let $[\lplus_{V'}]^\flat$ be the subspace of $\lplus_{V'} $ spanned by spiders where the label $b$ appears exactly once and $b^*$ does not appear at all.  The subspaces $[\lplus_{V'}^{\text{ab}}]^\flat$ and $\hairy_{V'}^\flat$ are defined similarly.

\begin{lemma} The map $\beta\colon  L_V \to \lplus_{V'}$ which puts the label $b$ on the root induces a
surjection $\overline{L}_V\twoheadrightarrow[\lplus_{V'}^{\text{ab}}]^\flat$.
\end{lemma}
\begin{proof} The image of $\beta$ lies in $[\lplus_{V'}]^\flat$, and after abelianization we get a surjective map to $[\lplus_{V'}^{\text{ab}}]^\flat$.  Now $\lplus_V\cdot L_V$ is in the kernel of this map, since
acting by an element of $\lplus_V$ on $L_V$ corresponds to taking a commutator in $\lplus_{V'}$.
\end{proof}

We also have   a map $[\lplus_{V'}^{\text{ab}}]^\flat\to H_1(\hairy_{V'})^\flat$ induced by the trace.  Altogether we have the following chain of maps:
$$\left(\ext (\lplus_{\infty}^{\text{ab}})^*\otimes [H_1(\hairy_{V_\infty'}^\flat)]^*\right)
\to
\left(\ext (\lplus_{\infty}^{\text{ab}})^*\otimes [(\lplus^{\text{ab}}_{V'_{\infty}})^\flat]^*\right)
\to
\left(\ext (\lplus_\infty^{\text{ab}})^*\otimes \overline{L}_{\infty}^*\right)
\to
H^*(\lplus_\infty; L_\infty^*),$$
inducing
$$
P\left(\ext (\lplus_{\infty}^{\text{ab}})^*\otimes (H_1(\hairy_{V_\infty'}^\flat))^*\right)^{\SP}
\to P H_c^*(\lplus_\infty; L_\infty^*)^{\SP}
$$

As in the last section, we can identify pieces of $P\left(\ext (\lplus_{\infty}^{\text{ab}})^*\otimes (H_1(\hairy_{V_\infty'}^\flat))^*\right)^{\SP}$ to give us classes in $P H_c^*(\lplus_\infty; L_\infty^*)^{\SP}$ which correspond via Gray's theorem to cycles for the homology of $\Aut(F_n)$.  We illustrate this  in the following theorem.

\begin{theorem}\label{thm:autclasses}
There is a series of cycles $e_{4k+3}\in Z_{4k+3}(\Aut(F_{2k+3});\F)$ for $k\geq 1$. \end{theorem}

\begin{proof}
Assume, to begin with, that $V$ (and therefore $V'$) is finite dimensional.
We first identify some convenient submodules of $H_1(\hairy_{V'}^\flat)$. In particular we look at the part of $\hairy_{1,2}$ with $2k$ hairs. In Theorem~\ref{thm:modular}, it is shown that the module $\SF{(2k-1,1)}(V')$ appears with multiplicity $s_{2k}+1$, where the ``+1" is contributed by Eisenstein series.
 Let $\SF{(2k-1,1)}^\flat(V')\subset \SF{(2k-1,1)}(V')$ be the subspace generated by tensors where the vector $b$ appears exactly once and its dual $b^*$ does not appear at all. We claim that $\SF{(2k-1,1)}^\flat(V')\cong V\otimes S^{2k-2}V$. One way to see this is to get an explicit description of $\SF{(2k-1,1)} (V')$
via the exact sequence $$0\to S^{2k}V'  \to V\otimes S^{2k-1}V'\to \SF{(1,2k-1)}(V')\to 0.$$
The left-hand map is defined by $v_1\otimes\cdots\otimes v_{2k}\mapsto \sum_{i=1}^{2k} v_i\otimes v_1\cdots \hat{v_i}\cdots v_{2k}.$ Any element in $\SF{(2k-1,1)}^\flat(V')$ can be represented by an element  $v_0\otimes v_1\cdots v_{2k-1}$, where exactly one of the $v_i$ is equal to $b$. If $v_0=b$, then modulo the image of the left-hand map, it can be rewritten as a sum of terms where the $b$ has moved to the right of the tensor. Thus $\SF{(2k-1,1)}^\flat(V')$ is spanned by elements $v_0\otimes bv_1\cdots v_{2k-2}$, where $v_i\in V$. These form a space isomorphic to $V\otimes S^{2k-2}V$ as claimed.

Let $S^{2k-1}V$ be the summand of $\lplus_V^{\rm{ab}}$ whose generators are represented by an oriented loop with $2k-1$ hairs attached. Then $[S^{2k-1}V\otimes\SF{(2k-1,1)}^\flat(V')]^{\SP}\cong [S^{2k-1}V\otimes (V\otimes S^{2k-2}V)]^{\SP}\cong \F$.
Now taking the inverse limit as the dimension of $V$ increases, we get a 1-dimensional subspace of
$(\lplus_{\infty}^{\text{ab}})^*\otimes (H_1(\hairy_{V_\infty'})^\flat)^*$.
This is represented by pairing the hairs of the loop with $2k-1$ hairs to the hairs of  a theta graph with one basepoint hair and $2k-1$ other hairs (see Figure~\ref{Eseven} for $k=2$). This graph is connected so represents a primitive class, which we denote $e_{4k-1}$, for $k\geq 2$.
\end{proof}
\begin{figure}
\begin{center}
\ifpdf\includegraphics[width=2in]{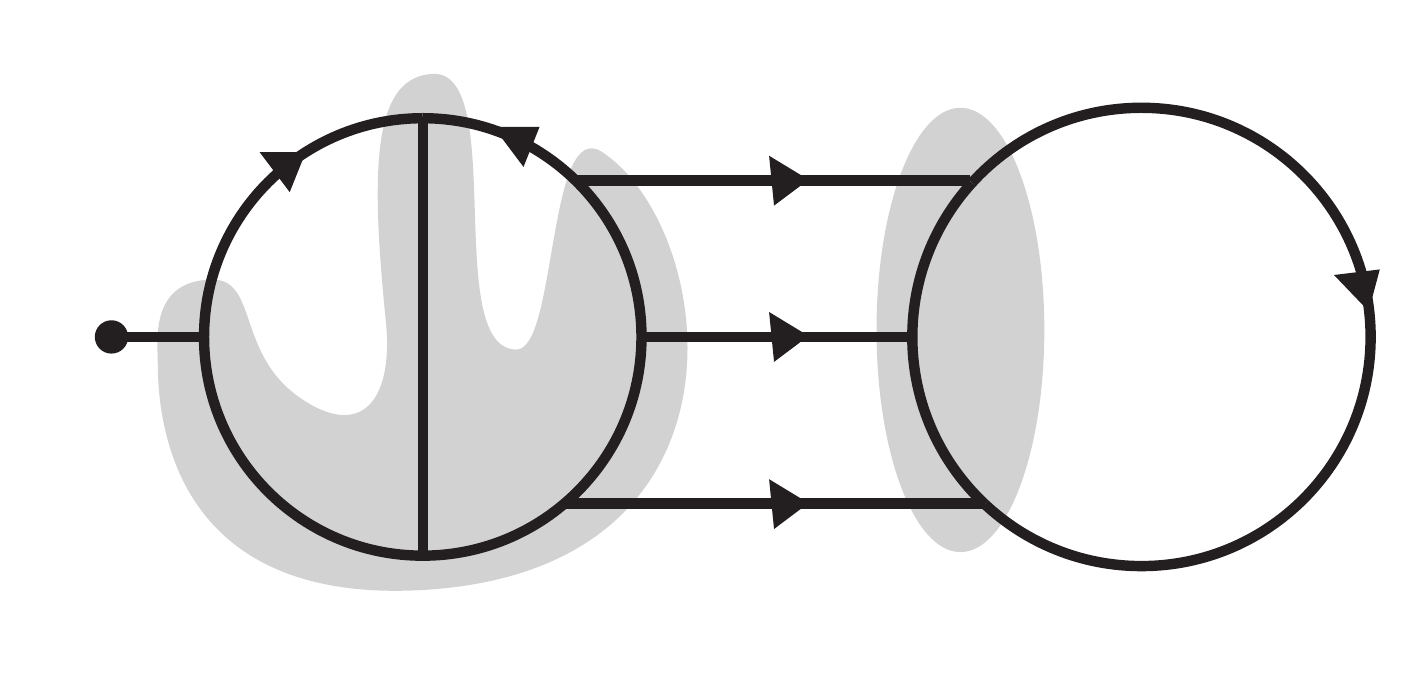}\fi
\caption{Hairy Lie graph representation of $e_7$}\label{Eseven}
\end{center}
\end{figure}

Computer calculations show that the first two cycles $e_7$ and $e_{11}$ are nontrivial in homology.  This brings the total list of known nontrivial rational homology groups  for $\Aut(F_n)$ and $\Out(F_n)$ to: $H_4(\Out(F_4);\mathbb Q)$, $H_8(\Out(F_6);\mathbb Q)$, $H_{12}(\Out(F_8);\mathbb Q)$, $H_4(\Aut(F_4);\mathbb Q)$, $H_7(\Aut(F_5);\mathbb Q)$ and $H_{11}(\Aut(F_7);\mathbb Q)$. The first three classes are part of Morita's original series. The fact that $H_{12}(\Out(F_8);\mathbb Q)\neq 0$ was recently proven by Gray~\cite{Gray}. The fact that $H_7(\Aut(F_5);\mathbb Q)\neq 0$ was proven by Gerlits~\cite{Gerlits}, though the interpretation in terms of the Eisenstein series is new.

\begin{remark}
In all of these cases, except $H_{12}(\Out(F_8))$ and $H_{11}(\Aut(F_7))$ which are unknown, computer calculations due to Gerlits and Ohashi~\cite{Gerlits, Ohashi} show that the homology spaces are one dimensional, so that these classes generate everything.
\end{remark}

\section{Hairy Lie graphs and automorphisms of punctured 3-manifolds }\label{sec:thornedgraphs}
Let $M_{n,s}$ be the compact $3$-manifold obtained from the connected sum of $n$ copies of $S^1 \times S^2$ by deleting the interiors of $s$ disjoint balls. In   \cite{HV2}  the group $\Gamma_{n,s}$  is defined to be the quotient of the mapping class group of $M_{n,s}$ by the normal subgroup generated by Dehn twists along embedded $2$-spheres. By a theorem of Laudenbach~\cite{Laudenbach}
$\Gamma_{n,0}\cong \Out(F_n)$ and $\Gamma_{n,1}\cong \Aut(F_n)$.

Hairy graph homology is related to the groups $\Gamma_{n,s}$ as follows.
Let $\hairy^{n,s}_V$ be the part of the hairy Lie graph complex  generated by connected graphs of rank $n$ with $s$ hairs.

\begin{theorem}\label{thm:Aut-Out}
There are isomorphisms
$$H_k(\hairy^{n,s}_V)\cong  H^{2n+s-2-k}(\Gamma_{n,s};V^{\otimes s})_{\Sigma_s}\cong H^{2n+s-2-k}(\Gamma_{n,s};\F)\otimes_{\Sigma_s}V^{\otimes s},$$
where the symmetric group $\Sigma_s$ acts simultaneously on $C^*(\Gamma_{n,s})$ and $V^{\otimes s}$.
\end{theorem}

When $s=0$, we recover the isomorphism $H_k(\hairy^{n,0})\cong H^{2n-2-k}(\Out(F_n);\F)$ described in~\cite{CV}, since the $0$-hair part of the hairy graph complex is just the Lie graph complex. When $s=1$, we get
$H_k(\hairy^{n,1}_V)\cong H^{2n-1-k}(\Aut(F_n);\F)\otimes V$.

\begin{proof}[Proof of Theorem~\ref{thm:Aut-Out}]
This is a straightforward adaptation of the proof for $s=0$, using Proposition~\ref{brown} and the spaces $A_{n,s}$ defined in~\cite{HV2} in place of Outer space. What we are calling hairs correspond to ``thorns" in \cite{HV2}. The only wrinkle is that in hairy graph homology hairs are labeled by elements of $V$ and do not come with a distinguished ordering, as in the definition of $A_{n,s}$. Hence to get an equality, we need to take the coinvariants under the action of the symmetric group which permutes the ordering on the thorns. This gives an isomorphism of $H_*(\hairy^{n,s}_V)$ with $H_*([C^*(\Gamma_{n,s})\otimes V^{\otimes s}]_{\Sigma_s})$ with some degree shift. Over the rationals, taking coinvariants commutes with homology. So we get an isomorphism $H_*([C^*(\Gamma_{n,s})\otimes V^{\otimes s}]_{\Sigma_s})\cong H^*(\Gamma_{n,s};\F)\otimes_{\Sigma_s} V^{\otimes s}$. On the other hand $C^*(\Gamma_{n,s})\otimes V^{\otimes s}\cong \Hom(C_*(\Gamma_{n,s}),V^{\otimes s})$, so we get $H_*([C^*(\Gamma_{n,s})\otimes V^{\otimes s}]_{\Sigma_s})\cong H_*(\Hom(C_*(\Gamma_{n,s}),V^{\otimes s}))_{\Sigma_s}=H^*(\Gamma_{n,s};V^{\otimes s})_{\Sigma_s}$.
\end{proof}


\begin{thebibliography}{1}

\bibitem{brown}
Ken Brown.
\newblock  Cohomology of Groups
\newblock {\em Graduate Texts in Mathematics}, 87. Springer-Verlag, New York-Berlin, 1982. x+306 pp. ISBN: 0-387-90688-6

\bibitem{Jim}
James Conant.
\newblock Ornate necklaces and the homology of the genus one mapping class
  group.
\newblock {\em Bull. London Math. Soc.}, 39(6):881--891, 2007.

\bibitem{CKV2} James Conant, Martin Kassabov and Karen Vogtmann.
\newblock On hairy graph homology.   
\newblock In preparation.

\bibitem{CST} James Conant, Rob Schneiderman and Peter Teichner.
\newblock Geometric filtrations of string links and homology cylinders.
\newblock arXiv:1202.2482v1  [math.GT]

\bibitem{CV}
James Conant and Karen Vogtmann.
\newblock On a theorem of Kontsevich.
\newblock {\em Algebr. Geom. Topol}, 3:1167--1224, 2003.



\bibitem{CVMorita}
James Conant and Karen Vogtmann.
\newblock Morita classes in the homology of automorphism groups of free groups.
\newblock {\em Geom. Topol.}, 8:1471--1499 (electronic), 2004.

\bibitem{CV86}
Marc Culler and Karen Vogtmann.
\newblock Moduli of graphs and automorphisms of free groups.
\newblock {\em Invent. Math.} 84 (1986), no. 1, 91--119.
\newblock


\bibitem{FH} William Fulton and Joe Harris.
\newblock Representation theory. A first course.
\newblock Graduate Texts in Mathematics, 129. Readings in Mathematics. Springer-Verlag, New York, 1991. xvi+551 pp. ISBN: 0-387-97527-6

\bibitem{Galatius} Soren Galatius.
\newblock Stable Homology of Automorphism Groups of Free Groups.
\newblock  {\em Ann. of Math.} (2) 173 (2011), no. 2, 705--768. 

\bibitem{GL} Stavros Garoufalidis and Jerome Levine.
\newblock Tree-level invariants of $3$--manifolds, Massey products and the Johnson homomorphism.
\newblock Graphs and patterns in mathematics and theoretical physics, volume 73
of Proc. Sympos. Pure Math., Amer. Math. Soc., Providence, RI, (2005) 173--203.


\bibitem{Gerlits} Ferenc Gerlits.
\newblock Invariants in chain complexes of graphs
\newblock Ph.D. thesis, Cornell University 2002

\bibitem{GW}
Roe Goodman and Nolan Wallach.
\newblock Representations and invariants of the classical groups.
\newblock {\em Encyclopedia of Mathematics and its Applications}, 68. Cambridge University Press, Cambridge, 1998. xvi+685 pp. ISBN: 0-521-58273-3

\bibitem{Gray} Jonathan Gray.
\newblock On the homology of automorphism groups of free groups
\newblock {Ph.D thesis, University of Tennessee 2011}

\bibitem{Hab} Klaus Haberland.
\newblock Perioden von Modulformen einer Variabler and Gruppencohomologie I (German) [Periods of modular forms of one variable and group cohomology I]
\newblock {\em Math. Nachr.} 112 (1983), 245--282.


\bibitem{HV}
Allen Hatcher and Karen Vogtmann.
\newblock Rational homology of $\Aut (F_n)$.
\newblock {\em Mathematical Research Letters}, 5:759--780, 1998.

\bibitem{HV2}
Allen Hatcher and Karen Vogtman.
\newblock Homology stability for outer automorphism groups of free groups.
\newblock {\em Algebraic and Geometric Topology}, 4:1253--1272, 2004


\bibitem{J}  Dennis Johnson.
\newblock  A survey of the Torelli group.
\newblock {|em Contemporary Math}, 20:163--179, 1983


\bibitem{Ko2}
Maxim Kontsevich.
\newblock Formal (non) commutative symplectic geometry.
\newblock {\em The Gelfand Mathematical Seminars, 1990--1992}, pages 173--187.

\bibitem{Ko1}
Maxim Kontsevich.
\newblock Feynman diagrams and low-dimensional topology.
\newblock {\em First European Congress of Mathematics}, 2:97--121, 1992.

\bibitem{L}  Jerome Levine.
\newblock  Homology cylinders: an enlargement of the mapping class group.
\newblock {\em Alg. and Geom. Topology }, 1: 243--270, 2001.

\bibitem{Laudenbach}
 Fran\c cois Laudenbach.
\newblock Sur les 2-sph\`eres dÕune vari\`et\'e de dimension 3.
\newblock {\em Ann. of Math.} 97 (1973) 57Ð81


\bibitem{MW} Ib Madsen and Michael Weiss.
\newblock The stable moduli space of Riemann surfaces: Mumford's conjecture.
\newblock Ann. of Math. (2)  165  (2007),  no. 3, 843--941.

\bibitem{Mor} Shigeyuki Morita.
\newblock Abelian quotients of subgroups of the mapping class group of surfaces.
\newblock {\em Duke Math Journal} 70:699--726, 1993.

\bibitem{Morita}
Shigeyuki Morita.
\newblock Structure of the mapping class groups of surfaces: a survey and a
  prospect.
\newblock 2:349--406 (electronic), 1999.

\bibitem{Morita2}
Shigeyuki Morita.
\newblock Lie algebras of symplectic derivations and cycles on the moduli spaces.
\newblock {\em Geometry \& Topology Monographs} 13 (2008) 335--354.

\bibitem{mss} Shigeyuki Morita, Takuya Sakasai, and Masaaki Suzuki.
\newblock Abelianizations of derivation Lie algebras of free associative algebra  and free Lie algebra.
\newblock arXiv:1107.3686v1

\bibitem{Ohashi} Ryo Ohashi.
\newblock The rational homology group of ${\rm Out}(F_n)$  for $n\leq 6$.
\newblock Experiment. Math. 17 (2008), no. 2, 167--179.


\end{thebibliography}
\end{document}